\documentclass[11pt]{amsart}
\usepackage{a4wide,amsfonts}
\usepackage{amsthm}
\usepackage{amsmath}
\usepackage{mathrsfs} 
\usepackage{amssymb}
\usepackage{amscd}
\usepackage[all]{xy}
\usepackage{hyperref}

\numberwithin{equation}{section}
\usepackage{epstopdf}
\usepackage{enumerate}
\newtheorem{theorem}{Theorem}[section]
\newtheorem{lemma}[theorem]{Lemma}
\newtheorem{corollary}[theorem]{Corollary}
\newtheorem{proposition}[theorem]{Proposition}

\newtheorem{conjecture}[theorem]{Conjecture}
\newtheorem{problem}[theorem]{Problem}

\theoremstyle{definition}

\newtheorem{definition}[theorem]{Definition}
\newtheorem{example}[theorem]{Example}

\newcommand{\IR}{\mathbb R}
\newcommand{\IN}{\mathbb N}

\newcommand{\w}{\omega}
\newcommand{\K}{\mathcal K}
\newcommand{\F}{\mathcal F}
\newcommand{\Ra}{\Rightarrow}
\newcommand{\U}{\mathcal U}

\newcommand{\MOP}{\mathsf{MOP}}
\newcommand{\DMOP}{\mathsf{DMOP}}
\newcommand{\WDMOP}{\mathsf{WDMOP}}

\newcommand{\GKF}{\mathsf{G_{KF}}}
\newcommand{\GNE}{\mathsf{G_{NE}}}
\newcommand{\GEN}{\mathsf{G_{EN}}}

\newcommand{\e}{\varepsilon}

\newcommand{\pr}{\mathrm{pr}}

\begin{document}

\title{On Baire category properties of function spaces $C_k'(X,Y)$}

\author{Taras Banakh}
\address{Institute of Mathematics, Jan Kochanowski University in Kielce (Ukraine), and
  Ivan Franko National University in Lviv (Ukraine)}
\curraddr{}
\email{t.o.banakh@gmail.com}
\thanks{}

\author{Leijie Wang}
\address{Department of Mathematics, Shantou University, Shantou, Guangdong,
515063, PR China}
\email{16ljwang@stu.edu.cn}
\thanks{}

\subjclass[2000]{Primary 54C35; Secondary 54E52}

\date{}

\dedicatory{}

\keywords{Function space, compact-open topology, Moving Off Property, meager, Baire}

\begin{abstract}
We prove that for a stratifiable scattered space $X$ of finite scattered height, the function space $C_k(X)$ endowed with the compact-open topology is Baire if and only if $X$ has the Moving Off Property of Gruenhage and Ma. As a byproduct of the proof we establish many interesting Baire category properties of the function spaces $C_k'(X,Y)=\{f\in C_k(X,Y):f(X')\subset\{*_Y\}\}$, where $X$ is a topological space, $X'$ is the set of non-isolated points of $X$, and $Y$ is a topological space with a distinguished point $*_Y$. 
\end{abstract}

\maketitle

\tableofcontents

\section{Introduction and Main Results}

This paper was motivated by the problem of characterization of scattered topological spaces $X$ whose function space $C_k(X)$ is Baire. Here $C_k(X)$ is the space of real-valued continuous functions on $X$, endowed with the compact-open topology.

A topological space $X$ is {\em Baire} if for any sequence $(U_n)_{n\in\w}$ of open dense sets in $X$, the intersection $\bigcap_{n\in\w}U_n$ is dense in $X$. In \cite{GMa} Gruenhage and Ma made a conjecture that for a Tychonoff space $X$,  the function space $C_k(X)$  is Baire if and only if $X$ has the Moving Off Property (abbreviated $\MOP$), which is defined as follows.

A family $\F$ of subsets of a topological space $X$ is called
\begin{itemize}
\item {\em discrete} if each point $x\in X$ has a neighborhood $O_x\subset X$ that meets at most one set of the family $\F$;
\item {\em strongly discrete} if each set $F\in\F$ has an open neighborhood $O_F\subset X$ such that the indexed family $\{O_F\}_{F\in\F}$ is discrete in the sense that each point $x\in X$ has a neighborhood $O_x\subset X$ that meets at most one set $O_F$, $F\in\F$;
\item {\em a moving off family} if for any compact set $K\subset X$ there exists a non-empty set $F\in\F$ such that $K\cap F=\emptyset$.
\end{itemize}

A topological space $X$ is defined to have the {\em Moving Off Property} (abbreviated $\MOP$) if each moving off family $\F$ of compact sets in $X$ contains an infinite strongly discrete subfamily $\mathcal D\subset\F$.

In Theorem 2.1 of \cite{GMa} Gruenhage and Ma observed that a Tychonoff space $X$ has $\MOP$ if its function space $C_k(X)$ is Baire, and made the following conjecture (see also \cite[Question 4.7]{Gru}).

\begin{conjecture}\label{h:GMa} A Tychonoff space $X$ has $\MOP$ if and only if its function space $C_k(X)$ is Baire.
\end{conjecture}

In \cite{GMa} this conjecture was confirmed for all $q$-spaces, i.e., spaces whose any point $x\in X$ admits a sequence $(U_n)_{n\in\w}$ of neighborhoods  such that each sequence $(x_n)_{n\in\w}\in\prod_{n\in\w}U_n$ has an accumulation point in $X$. The class of $q$-spaces includes all locally compact and all first-countable spaces. In \cite{GG} Conjecture~\ref{h:GMa} has been confirmed for subspaces of linearly ordered spaces. More extensive information on $\MOP$ can be found in the Ph.D. dissertation of Hughes \cite{Hughes}. Some set-theoretical questions related to $\MOP$ were studied by Tall  in \cite{Tall}.

In this paper we confirm Conjecture~\ref{h:GMa} for stratifiable scattered spaces of finite scattered height. Let us recall that a topological space $X$ is {\em scattered} if each non-empty subspace of $X$ contains an isolated point. A point $x$ of a topological space $X$ is called {\em isolated} if its singleton $\{x\}$ is clopen in $X$.
A subset of a topological space is {\em clopen} if it is both closed and open. For a topological space $X$ by $\dot X$ we denote the set of isolated points in $X$.

The complexity  of a scattered space $X$ can be measured by an ordinal number $\hbar(X)$, called the scattered height of $X$. It is defined as follows.

For a subspace $A\subset X$ of $X$ denote by $A':=A\setminus\dot A$ the set of non-isolated points of $A$. Let $X^{[0]}=X$ and for any ordinal $\alpha>0$ define the $\alpha$-th derived set $X^{[\alpha]}$ of $X$ by the recursive formula $$X^{[\alpha]}=\bigcap_{\beta<\alpha}(X^{[\beta]})'.$$The smallest ordinal $\alpha$ with $X^{[\alpha]}=X^{[\alpha+1]}$ is called the {\em scattered height} of $X$ and is denoted by $\hbar[X]$.
Observe that a topological space $X$ is scattered if and only if $X^{[\hbar[X]]}=\emptyset$.

A regular topological space $X$ is called {\em stratifiable} if to each point $x\in X$ it is possible to assign a decreasing sequence $(U_n(x))_{n\in\w}$ of neighborhoods such that each closed set $F$ is equal to the intersection $\bigcap_{n\in\w}\overline{U_n[F]}$ where $U_n[F]=\bigcup_{x\in F}U_n(x)$. The class of stratifiable spaces includes all metrizable spaces and has many nice properties, see \cite[\S5]{Grue}. In Theorem~\ref{t:ss} we shall prove that a scattered space $X$ of finite scattered height is stratifiable if and only if for every $n<\hbar[X]$ the set $X^{[n]}$ is a retract of $X$ and $X^{[n]}$ is a $G_\delta$-set in $X$.

One of the principal results of this paper is the following theorem.

\begin{theorem}\label{t:main} For a stratifiable space $X$, the function space $C_k(X)$ is Baire if and only if $X$ has $\MOP$ and the function space $C_k(X')$ is Baire.
\end{theorem}

Applying this theorem by induction, we get the following corollary.

\begin{corollary}\label{c:main} For a stratifiable space $X$, the
space $C_k(X)$ is Baire if and only if $X$ has $\MOP$ and for some $n\in\w$ the function space $C_k(X^{[n]})$ is Baire.
\end{corollary}

In its turn, Corollary~\ref{c:main} implies

\begin{corollary}\label{c2} For a stratifiable scattered space $X$ of finite scattered height, the
space $C_k(X)$ is Baire if and only if $X$ has $\MOP$.
\end{corollary}

Theorem~\ref{t:main} and Corollary~\ref{c:main} will be proved in Section~\ref{s5} after extensive preliminary work made in Sections~\ref{s2}--\ref{s11}. Namely, in Section~\ref{s2} we introduce a discrete version of $\MOP$, called $\DMOP$ and prove a game characterization of $\DMOP$, resembling the game characterization of $\MOP$, found by Gruenhage and Ma in \cite{GMa} (our game $\GKF(X)$ resembles also the games $\Gamma_1(X)$ and $\Gamma_2(X)$, studied by McCoy and Ntantu \cite[\S8]{McN}). In Section~\ref{s4} we apply the game characterization of $\DMOP$ to prove that a topological space $X$ has $\DMOP$ if and only if for any Polish space $Y$ with a distinguished point $*_Y$ the function space
$$C_k'(X,Y)=\big\{f\in C_k(X,Y):f(X')\subset\{*_Y\}\big\}$$ is Baire. In Section~\ref{s:W}, we study a ``winning'' modification of $\DMOP$, called $\WDMOP$ and prove that a topological space $X$ has $\WDMOP$ if and only if for any pointed Polish space $(Y,*_Y)$ the function space $C_k'(X,Y)$ is Choquet. In Sections~\ref{s:M} and \ref{s:CM} we characterize topological spaces $X,Y$ for which the function space $C_k'(X,Y)$ is metrizable or complete-metrizable. In Section~\ref{s:N} we characterize topological spaces $X$ for which the function space $C_k'(X,Y)$ has a countable network for any second-countable pointed space $Y$.
 As an application of the mentioned characterizations of Baire category properties of function spaces $C_k'(X,Y)$, in Section~\ref{s:a} we prove the following dichotomy: if for a topological space $X$ and a pointed Polish space $Y$ the function space $C_k'(X,Y)$ is analytic, then it is either Polish or meager (more precisely, $\infty$-meager). Also we prove that for a regular pointed space $Y$ the function space $C'_k(X,Y)$ is analytic if and only if it is cosmic and the function space $C_p'(X,Y)$ is analytic.

 The equivalence of meager and $\infty$-meager properties in the function spaces $C_k'(X,Y)$ is established in Section~\ref{s:mm}. In Section~\ref{s:ccm} we observe that for any topological space $X$ and pointed topological space $Y$, the function space $C_k'(X,Y)$ admits a continuous bijective map onto the function space $C_k'(X/X',Y)$ defined on the quotient space $X/X'$ having a unique non-isolated point, and give conditions under which this bijective map $C_k'(X,Y)\to C_k(X/X',Y)$ is (or is not) is a homeomorphism. In Section~\ref{s:ss} we prove that a scattered space $X$ of finite scattered height is stratifiable if and only if for every $k<\hbar[X]$ the set $X^{[k]}$ is a $G_\delta$-retract in $X$.
 In Section~\ref{s11} we study function spaces with values in rectifiable spaces $Y$ and prove that if $X'$ is a retract of $X$, then the function space $C_k(X,Y)$ is homeomorphic to $C_k(X',Y)\times C_k'(X,Y)$.   In the final section~\ref{s5} we return back to studying the function spaces $C_k(X,Y)$ over stratifiable scattered spaces with values in rectifiable spaces and using the game characterization of $\DMOP$, prove a more general version of Theorem~\ref{t:main} announced in the introduction. It should be mentioned that the results on Baire category properties of the function spaces $C_k'(X,Y)$ obtained in this paper are essentially used in our forthcoming paper \cite{BW},  devoted to studying the function spaces $C_{{\downarrow}\mathsf F}(X,Y)$ endowed with the Fell hypograph topology.

The characterization theorems proved in Sections~\ref{s4}--\ref{s:mm} can be unified in the following

\begin{theorem}\label{t:super} Let $Y$ be a pointed metrizable space containing more than one point and $X$ be a topological space containing an isolated point.
\begin{enumerate}
\item[\textup{(1)}] If $X$ does not have $\DMOP$, then $C_k'(X,Y)$ is meager and $\infty$-meager.
\item[\textup{(2)}] If $X$ has $\DMOP$ and $Y$ is Choquet, then the space $C_k'(X,Y)$ is Baire.
\item[\textup{(3)}] $C_k'(X,Y)$ is Choquet if and only if $Y$ is Choquet and $X$ has $\WDMOP$.
\item[\textup{(4)}] $C_k'(X,Y)$ is complete-metrizable if and only if $Y$ is complete-metrizable and $X$ is a $\dot\kappa$-space.
\item[\textup{(5)}] $C_k'(X,Y)$ is almost complete-metrizable if and only if $Y$ is Choquet and $X$ is a $\dot\kappa$-space.
\item[\textup{(6)}] $C_k'(X,Y)$ is Polish if and only if $Y$ is Polish and $X$ is a $\dot\kappa$-space with countable set $\dot X$ of isolated points.
\item[\textup{(7)}] $C_k'(X,Y)$ is almost Polish if and only if $Y$ is almost Polish and $X$ is a $\dot\kappa$-space with countable set $\dot X$.
\item[\textup{(8)}] $C_k'(X,Y)$ is metrizable if and only if  $X$ is a hemi-$\dot\kappa_\w$-space.
\item[\textup{(9)}] $C_k'(X,Y)$ is metrizable and separable if and only if $Y$ is separable and $X$ is a hemi-$\dot\kappa_\w$-space with countable set $\dot X$.
\item[\textup{(10)}] $C_k'(X,Y)$ has a countable network if and only if $Y$ is separable and $X$ has a countable $\dot\kappa$-network.
\end{enumerate}
\end{theorem}

The statements (1)--(10) of this theorem are proved in Theorems~\ref{t:m=M}, \ref{t:B}, \ref{t:C}, \ref{t:CM}, \ref{t:ACM}, \ref{t:P}, \ref{t:AP}, \ref{t:M}, \ref{t:MS}, \ref{t:N}, respectively. All undefined notions appearing in Theorem~\ref{t:super} can be found in the corresponding sections.

\section{The discrete moving off property}\label{s2}

In this section we discuss the discrete modification of $\MOP$, called $\DMOP$.

\begin{definition} A topological space $X$ has the {\em discrete moving off property} (briefly, $\DMOP$) if any moving off family $\F$ of  finite subsets of $\dot X$ contains an infinite subfamily $\mathcal D\subset\mathcal F$, which is discrete in $X$.
\end{definition}

We recall that $\dot X$ denotes the set of all isolated points in $X$, i.e. points $x\in X$ whose singleton $\{x\}$ is clopen in $X$.

It is clear that each space with $\MOP$ has $\DMOP$.

\begin{definition} A topological space $X$ is defined to have the {\em property of discrete diagonalization} if for any sequence $\F_0,\F_1,\F_2,\dots$ of infinite discrete families $\F_n$ of finite subsets of $\dot X$ there exists a sequence $(F_n)_{n\in\w}\in\prod_{n\in\w}\F_n$ such that the indexed family $\{F_n\}_{n\in\w}$ is discrete in $X$.
\end{definition}

\begin{lemma}\label{l:diag} Each topological space $X$ with $\DMOP$ has the property of discrete diagonalization.
\end{lemma}

\begin{proof}  Assume that $X$ has $\DMOP$ and let  $(\F_n)_{n\in\w}$ be a sequence of infinite discrete families of finite subsets of $\dot X$. Fix a countable subfamily $\{D_n\}_{n\in\w}\subset\F_0$ consisting of pairwise disjoint non-empty finite sets $D_n$ in $\dot X$. For every $n\in\IN$ let
$$\mathcal E_n=\big\{D_n\cup F_1\cup\dots\cup F_n:(F_i)_{i=1}^n\in\prod_{i=1}^n\F_i\big\}.$$ It is easy to see that $\mathcal E:=\bigcup_{n=1}^\infty\mathcal E_n$ is a moving off collection of non-empty finite subsets of $\dot X$. By $\DMOP$, there exists an infinite discrete subfamily $\{E_k\}_{k\in\w}\subset\mathcal E$ consisting of pairwise disjoint sets $E_k$, $k\in\w$. For every $k\in\w$ find a number $n(k)\in\IN$ with $E_k\in\mathcal E_{n(k)}$ and observe that the correspondence $k\mapsto n(k)$ is injective (as the sets $E_k$ are pairwise disjoint). Replacing $(E_k)_{k\in\w}$ by a suitable subsequence, we can assume that the sequence $(n(k))_{k\in\w}$ is increasing and hence $n(k)\ge k$ for all $k\in\w$. Then for any $k\in\w$ there exists a set $F_k\in\F_k$ such that $F_k\subset E_k\in\mathcal E_{n(k)}$. The discreteness of the family $\{E_{k}\}_{k\in\w}$ implies that discreteness of the family $\{F_k\}_{k\in\w}$.
\end{proof}

Now we are going to present a game characterization of $\DMOP$ with help of the game $\GKF(X)$, played by two players, $\mathsf{K}$ and $\mathsf{F}$ on a topological space $X$. The player $\mathsf K$ starts the game. At the $n$-th inning the player $\mathsf K$ chooses a compact subset $K_n\subset X$  and player $\mathsf F$ responds by choosing a finite  subset $F_n$ of $\dot X$ such that $F_n\cap K_n=\emptyset$. At the end of the game, the player $\mathsf K$ is
declared the winner if the family $\{F_n\}_{n\in\IN}$ is discrete in $X$; otherwise the player $\mathsf F$ wins the game.

For a topological space $X$ by $\mathcal F(X)$ and $\K(X)$ we denote the families of finite and compact subsets of $X$, respectively.

For  set $A$ by $A^{<\w}$ we denote the family $\bigcup_{n\in\w}A^n$ of all finite sequences $(a_0,\dots,a_{n-1})$ of elements of $A$. The set $A^{<\w}$ is a tree with respect to the partial order $\le$ defined by $(a_0,\dots,a_n)\le (b_0,\dots,b_m)$ iff $n\le m$ and $a_i=b_i$ for $i\le n$. For a sequence $s=(a_0,\dots,a_{m-1})$ and a number $n\le m$ let $s{\restriction}n:=(a_0,\dots,a_{n-1})$ be the initial segment of $s$ of length $n$.


The following theorem is just a suitable modification of the game characterization of $\MOP$, proved by Gruenhage and Ma in  \cite[Theorem 2.3]{GMa}.

\begin{theorem}\label{t:game} A topological space $X$  has $\DMOP$ if and only if
the player $\mathsf F$ has no winning strategy in the game $\GKF(X)$.
\end{theorem}

\begin{proof} If $X$ does not have $\DMOP$, then there exists a moving off family $\F$ of finite  sets in $\dot X$ containing no infinite discrete subfamily. Then the player $\mathsf F$ can win the game $\GKF(X)$ by always choosing distinct members of $\F$.

Now assume that $X$ has $\DMOP$ and let $S_{\mathsf F}:\K(X)^{<\w}\to \mathcal F(\dot X)$ be any strategy of the player $\mathsf F$ in the game $\GKF(X)$.  The strategy $S_{\mathsf F}$ is a function  assigning to each finite sequence $(K_1,\dots,K_{n})\in\K(X)^{<\w}$ a finite set $S_{\mathsf F}(K_1,\dots,K_{n})\subset \dot X$ that is disjoint with the compact set $K_{n}$. Since the player $\mathsf K$ starts the game, we can assume that $S_{\mathsf F}$ assigns to the unique sequence $\K(X)^0$ of length zero the empty subset of $\dot X$.

We shall inductively construct a countable subtree $T\subset\K(X)^{<\w}$ such that for any sequence $(K_1,\dots,K_n)\in T$ the family $\{S_{\mathsf F}(K_1,\dots,K_n,K):(K_1,\dots,K_n,K)\in T\}$ is infinite and discrete.  The tree $T$ will be constructed as the union $T=\bigcup_{n\in\w}T_n$ of trees $T_n$ of height $n$. To start the inductive construction, we put $T_0:=\K(X)^0=\{\emptyset\}$. Assume that for some $n\in\w$ the subtree $T_n\subset\bigcup_{k\le n}\K(X)^{k}$ has been defined. For any node $(K_1,\dots,K_n)\in T_n$, consider the family $\{S_{\mathsf F}(K_1,\dots,K_n,K):K\in\K(X)\}$ and observe that it is moving off. Since $X$ has $\DMOP$, this family contains an infinite discrete subfamily $\{S_{\mathsf F}(K_1,\dots,K,K):K\in T_{n+1}(K_1,\dots,K_n)\}$ for some countable infinite subset $T_{n+1}(K_1,\dots,K_n)\subset\K(X)$. Finally, put 
$$T_{n+1}=T_n\cup\bigcup_{(K_1,\dots,K_n)\in T_n}\{(K_1,\dots,K_n,K):K\in T_{n+1}(K_1,\dots,K_n)\}.$$

 It is clear that the countable tree $T=\bigcup_{n\in\w}T_n$ has the required property: for every sequence $(K_1,\dots,K_n)\in T$ the indexed family $\{S_{\mathsf F}(K_1,\dots,K_n,K):(K_1,\dots,K_n,K)\in T\}$ is infinite and discrete.

By Lemma~\ref{l:diag}, the space $X$ has the discrete diagonalization property. So to each sequence $(K_1,\dots,K_n)\in T$ we can assign a compact set $K_{n+1}=\kappa(K_1,\dots,K_n)$ such that $(K_1,\dots,K_{n+1})\in T$ and the indexed family $$\mathcal D=\{S_{\mathsf F}(K_1,\dots,K_n,\kappa(K_1,\dots,K_n)):(K_1,\dots,K_n)\in T\}$$ is discrete.
Now consider the sequence $(K_n)_{n\in\w}\in\K(X)^\w$ defined recursively as $K_1=\kappa(\emptyset)$ and $K_{n+1}=\kappa(K_1,\dots,K_n)$ for $n\in\IN$.  Also put $F_n=S_{\mathsf F}(K_1,\dots,K_n)$ for every $n\in\IN$. Observe that the indexed family $\{F_n\}_{n\in\IN}$ is discrete, being a subfamily of the discrete family $\mathcal D$. Now we see that the sequence $$K_1,F_1,K_2,F_2,K_3,F_3,\dots$$ is a sequence of moves of the players $\mathsf K,\mathsf F$ in which the player $\mathsf F$ plays according to the strategy $S_{\mathsf F}$. Since the family $\{F_n\}_{n\in\IN}$ is discrete, the player $\mathsf K$ wins, which means that the strategy $S_{\mathsf F}$ of the player $\mathsf F$ is not winning.
\end{proof}

\begin{definition} A topological space $X$ is defined to have the {\em winning discrete moving off property} (abbreviated $\WDMOP$) if the player $\mathsf K$ has a winning strategy $S_{\mathsf K}$ in the game $\GKF(X)$.

This winning strategy is a function $S_{\mathsf K}:\F(\dot X)^{<\w}\to\K(X)$ assigning to any finite sequence $(F_0,\dots,F_{n-1})$ of finite subsets of $\dot X$ a compact subset $S_{\mathsf K}(F_0,\dots,F_{n-1})$ of $X$ such that a family $(F_n)_{n\in\w}$ of finite subsets of $\dot X$ is discrete in $X$ if for every $n\in\w$ the set $F_n$ is disjoint with the compact set $S_{\mathsf K}(F_0,\dots,F_{n-1})$.
\end{definition}

It is clear that $\WDMOP$ implies $\DMOP$.

\begin{lemma}\label{l:sW} If a topological space $X$ has $\WDMOP$, then the player $\mathsf K$ has a strategy $S_{\mathsf K}:\mathcal F(\dot X)^{<\w}\to\K(X)$ in the game $\GKF(X)$ such that for any sequence $(F_n)_{n\in\w}$ of finite subsets of $\dot X$ the indexed family $\{F_n\setminus S_{\mathsf K}(F_0,\dots,F_{n-1})\}_{n\in\w}$ is discrete in $X$.
\end{lemma}

\begin{proof} Since the space $X$ has $\WDMOP$, the player $\mathsf K$ has a winning strategy in the game $\GKF(X)$. This winning strategy is a function  $W:\F(\dot X)^{<\w}\to\K(X)$ such that a sequence $(F_n)_{n\in\w}$ of finite subsets of $\dot X$ is discrete in $X$ if $F_n\cap W(F_0,\dots,F_{n-1})=\emptyset$ for every $n\in\w$. Define a strategy $S_{\mathsf K}:\mathcal F(\dot X)^{<\w}\to\mathcal K(X)$ of the player $\mathsf K$ in the game $\GKF(X)$ assigning to each finite sequence $(F_0,\dots,F_{n-1})\in\mathcal F(\dot X)^{<\w}$ the last element $K_n$ of the sequence $(K_0,\dots,K_n)$ of compact subsets of $X$, defined by the  recursive formula: $K_{i}=W(F_0\setminus K_0,\dots,F_{i-1}\setminus K_{i-1})$ for $i\le n$.

We claim that the strategy $S_{\mathsf K}$ has the required property. Indeed, given any sequence $(F_n)_{n\in\w}$ of finite subsets of $\dot X$, define the sequence $(K_n)_{n\in\w}$ of compact subsets of $X$ by the recursive formula $K_{n}=W(F_0\setminus K_0,\dots,F_{n-1}\setminus K_{n-1})$ for $n\in\w$. The definition of the strategy $S_{\mathsf K}$ ensures that $K_{n}=S_{\mathsf F}(F_0,\dots,F_{n-1})$ for every $n\in\w$. For every $n\in\w$ consider the finite subset $E_n:=F_n\setminus K_n$ of $F_n\subset\dot X$ and observe that $E_n$ is disjoint with the compact set $K_n=W(E_0,\dots,E_{n-1})$. Since the strategy $W$ is winning, the indexed family $$(E_n)_{n\in\w}=(F_n\setminus K_n)_{n\in\w}=\big(F_n\setminus W(F_0\setminus K_0,\dots,F_{n-1}\setminus K_{n-1})\big)_{n\in\w}=\big(F_n\setminus S_{\mathsf K}(F_0,\dots,F_{n-1})\big)_{n\in\w}$$ is discrete in $X$.
\end{proof}

\section{A convenient base for the function space $C_k'(X,Y)$}\label{s3}

For topological spaces $X,Y$, let $C_k(X,Y)$ be the space of continuous functions from $X$ to $Y$, endowed with the compact-open topology. This topology is generated by the subbase consisting of the sets $$[K;U]:=\{f\in C_k(X,Y):f(K)\subset U\}$$where $K$ is a compact subset of $X$ and $U$ is an open subset of $Y$.

By a {\em pointed topological space} we understand a topological space $Y$ with a distinguished point $*_Y\in Y$.

A pointed topological space $Y$ is defined to be
\begin{itemize}
\item {\em $*$-first-countable} if $Y$ is first-countable at its distinguished point $*_Y$;
\item {\em $*$-admissible} if  the distinguished point $*_Y$ of $Y$ has a neighborhood $U_*$ which is not dense in $Y$.
\end{itemize}
It is easy to see that a pointed topological space $Y$ is $*$-admissible if $Y$ is Hausdorff and contains more than one point.

For a topological space $X$ and a pointed topological space $Y$ we shall study the Baire category properties of the subspace
$$C_k'(X,Y):=\big\{f\in C_k(X,Y):f(X')\subset\{*_Y\}\big\}\subset C_k(X,Y),$$where  $X':=X\setminus \dot X$ is the set of non-isolated points of $X$.

First we describe a convenient base of the  topology of the function space $C_k'(X,Y)$. For two sets $K\subset X$ and $U\subset Y$ we keep the notation
$$[K;U]:=\{f\in C_k'(X,Y):f(K)\subset Y\}.$$

 Fix a base $\mathcal B_Y\ni Y$ of the topology of the space $Y$ and consider the family $\mathcal Q$ of all quadruples $(K,U,F,u)$, where
\begin{itemize}
\item $K$ is a compact subset of $X$;
\item $U\in\mathcal B_Y$ is a neighborhood of the distinguished point $*_Y$ of $Y$;
\item $F\subset \dot X\setminus K$ is a finite set of isolated point in $X$;
\item $u:\dot X\to\mathcal B_Y$ is a function assigning to each point $x\in \dot X$ a non-empty basic open set $u(x)\in\mathcal B_Y$ such that $u(x)=Y$ if $x\notin F$.
\end{itemize}
For any quadruple $(K,U,F,u)\in\mathcal Q$, consider the open set
$$[K;U|F;u]:=[K;U]\cap\bigcap_{x\in F}[\{x\};u(x)]$$ in the function space $C_k'(X,Y)$, and observe that this open set is not empty.

The following lemma shows that the family $\{[K;U|F;u]:(K,U,F,u)\in\mathcal Q\}$ is a base of the topology of the function space $C_k'(X,Y)$.

\begin{lemma}\label{l:base}  For any quadruple $(K,U,F,u)\in\mathcal Q$, function $f\in C'_k(X,Y)$ and neighborhood $O_f\subset C_k'(X,Y)$ of $f$ there exists a quadruple $(\tilde K,\tilde U,\tilde F,\tilde u)\in\mathcal Q$ such that
$$F\subset\tilde F,\quad K\cup F\subset\tilde K\cup\tilde F,\quad\tilde U\subset U\quad\mbox{and}\quad f\in[\tilde K;\tilde U|\tilde F;\tilde u]\subset O_f.$$
\end{lemma}

\begin{proof} By  definition of the compact-open topology on $C_k'(X,Y)$, there are compact sets $K_1,\dots,K_n\subset X$ and open sets $U_1,\dots,U_n\subset X$ such that $f\in \bigcap_{i=1}^n[K_i;U_i]\subset O_f$.

Choose any basic open set $\tilde U\in\mathcal B_Y$ such that $$*_Y\in \tilde U\subset U\cap\bigcap\{U_i:1\le i\le n,\;*_Y\in U_i\}.$$
Consider the compact set $C=K\cup F\cup\bigcup_{i=1}^nK_i$ in $X$ and observe that the open set $V=(f{\restriction}C)^{-1}(\tilde U)$ contains the set $X'\cap C$.
The set $C\setminus V\subset\dot X$ is finite, being a closed discrete subspace of the compact space $C$. Then the union $\tilde F:=(C\setminus V)\cup F\subset \dot X$ is clopen in $X$ and its complement $\tilde K:=C\setminus\tilde F$ is a compact subset of $X$.

For every $x\in\dot X\setminus\tilde F$ put $\tilde u(x)=Y$ and for every $x\in \tilde F$ choose a basic neighborhood $\tilde u(x)\in\mathcal B_Y$ of the point $f(x)\in Y$ such that $\tilde u(x)\subset\bigcap\{U_i:1\le i\le n,\;f(x)\in U_i\}$. It is easy to see that $f\in[\tilde K;\tilde U|\tilde F;\tilde u]\subset \bigcap_{i=1}^n[K_i;U_i]\subset O_f$.
\end{proof}

As a first application of the above base, we find a conditions of the topological spaces $X,Y$ ensuring that the function space $C_k'(X,Y)$ has countable cellularity.

We recall that a topological space has {\em countable cellularity} if it contains no uncountable family of pairwise disjoint open sets.

\begin{proposition}\label{p:cellular} If the space $Y$ is second-countable and the set $\dot X$ of isolated points of $X$ is of type $F_\sigma$ in $X$, then the function space $C'_k(X,Y)$ has countable cellularity.
\end{proposition}

\begin{proof} To derive a contradiction, assume that $C'_k(X)$ contains an uncountable family $\{W_i\}_{i\in\w_1}$ of non-empty pairwise disjoint open sets. Fix a countable base $\mathcal B_Y$ of the topology of the space $Y$. By Lemma~\ref{l:base},   we can assume that each set $W_i$ is of basic form
$W_i=[K_i;U_i|F_i;u_i]$ for some quadruple $(K_i,U_i,F_i,u_i)\in\mathcal Q$. By the $\Delta$-Lemma \cite[9.18]{Jech}, there exists an uncountable subset $\Omega\subset\w_1$ and a finite set $F$ such that $F_i\cap F_j=F$ for any distinct elements $i,j\in\Omega$. By Pigeonhole Principle, for some function $u:F\to\mathcal B_Y$ the set $\{i\in \Omega:u_i{\restriction}F=u\}$ is uncountable. Replacing $\Omega$ by this uncountable set, we can assume that $u_i{\restriction}F=u$ for all $i\in\Omega$.

By our assumption, the set $\dot X$ is an $F_\sigma$-set in $X$, so $\dot X=\bigcup_{n\in\w}D_n$ for some increasing sequence $(D_n)_{n\in\w}$ of closed (discrete) sets $D_n\subset X$. By the Pigeonhole Principle, for some $n\in\w$ the set $\Omega_n:=\{i\in\Omega: F_i\subset D_n\}$ is uncountable. Since the set $D_n$ is closed and discrete in $X$, for every $i\in\Omega_n$ the compact set $E_i:=K_i\cap D_n$ is finite. By the $\Delta$-Lemma \cite[9.18]{Jech}, there exists an uncountable subset $\Lambda\subset\Omega_n$ and a finite set $E$ such that  $E_i\cap E_j=E$ for any distinct elements $i,j\in\Lambda$. Now take any element $i\in\Lambda$ and observe that $$F\cap E\subset F_i\cap E_i\subset F_i\cap K_i=\emptyset.$$  Since each of the families $(F_j\setminus F)_{j\in\Lambda}$ and $(E_j\setminus E)_{j\in\Lambda}$ is disjoint, and the sets $F_i,E_i$ are finite, the set $$\lambda:=\{j\in\Lambda:(F_j\setminus F)\cap E_i\ne\emptyset\}\cup\{j\in\Lambda:F_i\cap (E_j\setminus E)\ne\emptyset\}$$is finite. Take any element $j\in\Lambda\setminus(\lambda\cup\{i\})$ and observe that $$F_i\cap K_j=F_i\cap D_n\cap K_j=F_i\cap E_j=(F_i\cap (E_j\setminus E))\cup (F_i\cap E)\subset \emptyset\cup(F_i\cap E_i)=\emptyset.$$ By analogy we can check that $F_j\cap K_i=\emptyset$. This allows us to choose a (necessarily continuous) continuous function $f:X\to Y$ such that
\begin{itemize}
\item $f(x)\in u(x)$ for all $x\in F$;
\item $f(x)\in u_i(x)$ for all $x\in F_i\setminus F$;
\item $f(x)\in u_j(x)$ for all $x\in F_j\setminus F$;
\item $f(x)=*_Y$ for all $x\in X\setminus (F_i\cup F_j)$.
\end{itemize}
It is clear that $$f\in [K_i;U_i|F_i,u_i]\cap[K_j;U_j;F_j;u_j]=W_i\cap W_j,$$which is a desired contradiction.
\end{proof}

\section{A function space characterization of $\DMOP$}\label{s4}

In this section we shall characterize $\DMOP$  in terms of Baire category properties of function spaces $C_k'(X,Y)$.


We shall use the classical Oxtoby's characterizations of Baire and meager spaces in terms of the Choquet games $\GNE(X)$ and $\GEN(X)$, which are played by two players $\mathsf E$ and $\mathsf {N}$ (abbreviated from $\mathsf{Empty}$ and $\mathsf{Non}\mbox{-}\mathsf{Empty}$) on a topological space $X$.

The game $\GNE(X)$ is started by the player $\mathsf N$ whose chooses a non-empty open set $U_1$ of $X$. Then player $\mathsf E$ responds selecting a non-empty open set $V_1$. In the $n$th inning the player $\mathsf N$ selects a non-empty open set $U_n\subset V_{n-1}$ and the player $\mathsf E$ responds selecting a non-empty open set $V_n\subset U_n$. At the end of the game the player $\mathsf E$ is declared the winner if the intersection $\bigcap_{n\in\IN}U_n=\bigcap_{n\in\IN}V_n$ is empty; otherwise the player $\mathsf N$ wins the game $\GNE(X)$.

The game $\GEN(X)$ is started by the player $\mathsf E$ whose chooses any non-empty open set $U_1$ of $X$. Then player $\mathsf N$ responds selecting a non-empty open set $V_1$. In the $n$th inning the player $\mathsf E$ selects a non-empty open set $U_n\subset V_{n-1}$ and the player $\mathsf N$ responds selecting a non-empty open set $V_n\subset U_n$. At the end of the game the player $\mathsf E$ is declared the winner if the intersection $\bigcap_{n\in\IN}U_n=\bigcap_{n\in\IN}V_n$ is empty; otherwise the player $\mathsf N$ wins the game $\GEN(X)$.

The following classical characterization can be found in  \cite{Oxtoby}.

\begin{theorem}[Oxtoby]\label{t:Oxtoby} A topological space $X$ is
\begin{itemize}
\item meager if and only if the player $\mathsf E$ has a winning strategy in the game $\GNE(X)$;
\item Baire if and only if the player $\mathsf E$ has no winning strategy in the game $\GEN(X)$.
\end{itemize}
\end{theorem}
A topological space $X$ is defined to be {\em Choquet} if the player $\mathsf N$ has a winning strategy in the Choquet game $\GEN(X)$.

Oxtoby's Theorem~\ref{t:Oxtoby} implies that $$\mbox{Choquet $\Ra$ Baire $\Ra$ non-meager}.$$

By \cite[8.17]{Ke} (see also \cite[7.3]{BanP}), a metrizable topological space is Choquet if and only if it almost complete-metrizable;  moreover, any open continuous image of a Choquet space is Choquet and the Tychonoff product of any family of Choquet spaces is Choquet \cite{White}.
\smallskip

A topological space $X$ is defined to be
\begin{itemize}
\item {\em complete-metrizable} if $X$ is homeomorphic to a complete metric space;
\item {\em Polish} if $X$ is separable and complete-metrizable;
\item {\em almost Polish} if $X$ contains a dense Polish subspace;
\item {\em almost complete-metrizable} if $X$ contains a dense complete-metrizable subspace.
\end{itemize}

For every topological space we have the implications
$$\xymatrix{
\mbox{Polish}\ar@{=>}[d]\ar@{=>}[r]&\mbox{complete-metrizable}\ar@{=>}[d]&&\mbox{non-meager}\\
\mbox{almost Polish}\ar@{=>}[r]&\mbox{almost complete-metrizable}\ar@{=>}[r]&\mbox{Choquet}\ar@{=>}[r]&\mbox{Baire}\ar@{=>}[u]
}
$$

We recall that a pointed topological space $Y$ is $*$-admissible if  its distinguished point $*_Y$ has a neighborhood $U_*$ which is not dense in $Y$.

\begin{lemma}\label{l:B1} Let $X$ be a topological space and $Y$ be a $*$-admissible pointed topological space. If the function space $C_k'(X,Y)$ is non-meager, then the space $X$ has $\DMOP$.
\end{lemma}

\begin{proof} Assume that the function space $C_k'(X,Y)$ is non-meager. By Theorem~\ref{t:Oxtoby}, the player $\mathsf E$ has no winning strategy in the Choquet game $\GNE(C_k'(X,Y))$. By Theorem~\ref{t:game}, the $\DMOP$ for the space $X$ will follow as soon as we show that the player $\mathsf F$ has no winning strategy in the game $\GKF(X)$. 
Let $S_{\mathsf F}:\K(X)^{<\w}\to\mathcal F(\dot X)$ be any strategy of the player $\mathsf F$ in the game $\GKF(X)$. The strategy $S_{\mathsf F}$ assigns to each finite sequence of compact sets $(K_0,\dots,K_n)\in\K(X)^{<\w}$ a finite set $S_{\mathsf F}(K_0,\dots,K_n)\subset \dot X\setminus K_n$.

Define a new strategy $\tilde S_{\mathsf F}$ of the player $\mathsf F$ in the game $\GKF$ assigning to each sequence $(K_0,\dots,K_n)\in\K(X)^{<\w}$ the sequence $S_{\mathsf F}(\tilde K_0,\dots,\tilde K_n)$ where $\tilde K_0=K_0$ and $\tilde K_i=K_i\cup \bigcup_{j<i}\tilde S_{\mathsf F}(K_0,\dots,K_{j})$ for $1<i\le n$.

Now we use the strategy $\tilde S_{\mathsf F}$, to describe a strategy $S_{\mathsf E}$ of the player $\mathsf E$ in the Choquet game $\GNE(C_k'(X,Y))$.

Let $\tau$ denote the family of all non-empty open sets in $C_k'(X,Y)$. By our assumption, the pointed space $Y$ is $*$-admissible. So, there exists a non-empty open set $W\subset Y$ whose closure does not contain the distinguished point $*_Y$ of $Y$.

The definition of the compact-open topology ensures that for every $U\in \tau$ there exists a compact set $\kappa(U)\subset X$ such that for any finite set $F\subset \dot X\setminus \kappa(U)$ the open set $U\cap [F;W]$ is not empty.

For any sequence $(U_0,\dots,U_n)\in\tau^{<\w}$ let $S_{\mathsf E}(U_0,\dots,U_n)=U_n\cap [F_n;W]$ where $$F_n=\tilde S_{\mathsf F}(\kappa(U_0),\dots,\kappa(U_n))\subset \dot X\setminus \kappa(U_n)$$ is the answer of the player $\mathsf F$ to the moves $(\kappa(U_0),\dots,\kappa(U_n))$ of the player $\mathsf K$ in the game $\GKF(X)$, according to the strategy $\tilde S_{\mathsf F}$. Since $F_n\cap \kappa(U_n)=\emptyset$, the open set $U_n\cap[F_n;W]$ is non-empty and hence it is a legal move of the player $\mathsf E$ in the Choquet game $\GNE(C_k'(X,Y))$. Since the player $\mathsf E$ has no winning strategy in the game $\GNE(C_k'(X,Y))$, the strategy $S_{\mathsf E}$ is not winning. So there exists an infinite sequence $(U_n)_{n\in\w}\in\tau^\w$ such that  $U_{n+1}\subset S_{\mathsf E}(U_0,\dots,U_n)$ for all $n\in\IN$ and the intersection $\bigcap_{n=1}^\infty U_n$ is not empty and hence contains some function  $f\in C_k'(X,Y)$.

Let $F_0=\emptyset$ and for every $n\in\IN$ let $K_n=\kappa(U_n)\cup \bigcup_{i<n}F_{i}$ and $$F_n=\tilde S_{\mathsf F}(\kappa(U_1),\dots,\kappa(U_n))=S_{\mathsf F}(K_1,\dots,K_n)\subset \dot X\setminus  K_n\subset\dot X\setminus \bigcup_{i<n}F_i.$$

It follows that the family $(F_n)_{n\in\IN}$ is disjoint and $f\in \bigcap_{n\in\IN}[F_n;W]$. Since $f\big(\bigcup_{n\in\w}F_n\big)\subset W$ and $*_Y\notin \overline{W}$, the continuity of $f$ guarantees that the closure of the set $\bigcup_{n=1}^\infty F_n$ does not intersect the set $X'\subset f^{-1}(*_Y)$ and hence the disjoint family $\{F_n\}_{n\in\IN}$ is discrete in $X$.

Observe that  $$K_1,F_1,K_2,F_2,\dots$$is the sequence of the moves of the players $\mathsf K$ and $\mathsf F$ in the game $\GKF(X)$, where the player $\mathsf F$ plays according to the strategy $S_{\mathsf F}$ and eventually loses as the family $\{F_n\}_{n\in\IN}$ is discrete. So, the strategy $S_{\mathsf F}$ is not winning.
\end{proof}


\begin{lemma}\label{l:B2} Let $Y$ be a $*$-first-countable pointed Choquet space. If a topological  space $X$ has $\DMOP$, then the function space $C_k'(X,Y)$ is Baire.
\end{lemma}

\begin{proof}
By the $*$-first-countability of the pointed space $Y$, the distinguished point $*_Y$ of $Y$ has a countable neighborhood base $\{O_n\}_{n\in\w}$ such that $O_{n+1}\subset O_n$ for all $n\in\w$. Let $\tau_Y$ be the family of all non-empty open sets in $Y$. Since the space $Y$ is Choquet, the player $\mathsf N$ has a winning strategy $S_{\mathsf N}:\tau_Y^{<\w}\to\tau_Y$ in the Choquet game $\GEN(Y)$. The winning strategy $S_{\mathsf N}$ is a function assigning to each finite sequence $(U_0,\dots,U_n)$ of non-empty open sets in $X$ a non-empty open set $S_{\mathsf N}(U_0,\dots,U_n)\subset U_n$ such that for any infinite sequence $(U_n)_{n\in\w}\in\tau_Y^\w$ the intersection $\bigcap_{n\in\w}S_{\mathsf N}(U_0,\dots,U_n)$ is not empty if $U_n\subset S_{\mathsf N}(U_0,\dots,U_{n-1})$ for any $n\in\IN$. Since the player $\mathsf E$ starts the game $\GEN(Y)$, it is convenient to assume that $S_{\mathsf N}(\emptyset)=Y$ for the empty sequence $\emptyset\in \tau_Y^0$. We can also assume that $S_{\mathsf N}(U_0,\dots,U_{n-1})=Y$ if $U_i=Y$ for all $i<n$.

To derive a contradiction, assume that the function space $C_k'(X,Y)$ is not Baire. Then $C_k'(X,Y)$ contains a non-empty open meager subset $W$, which is contained in the countable union $\bigcup_{n\in\w}M_n$ of an increasing sequence $(M_n)_{n\in\w}$ of closed nowhere dense sets in $C_k'(X,Y)$.

Replacing $W$ by a smaller set and applying Lemma~\ref{l:base}, we can assume that $W$ is of basic form $$W=[\kappa_\emptyset;U_\emptyset|F_\emptyset;u_\emptyset]$$for some quadruple $(\kappa_\emptyset,U_\emptyset,F_\emptyset,u_\emptyset)\in\mathcal Q$ such that $U_\emptyset\subset O_0$. Here we assume that the base $\mathcal B_Y$ of the topology of $Y$ coincides with the family $\tau_Y$ of all non-empty open sets in $Y$.

For any $n\in\w$ and any finite sequence $s=(K_0,\dots,K_n)\in\K(X)^{n+1}$  we shall define inductively a quadruple $(\kappa_{s},U_s,F_s,u_s)\in\mathcal Q$ such that for the sequence $t=(K_0,\dots,K_{n-1})$ the following conditions are satisfied:
\begin{itemize}
\item[(a)] $F_t\subset F_s$;
\item[(b)] $\kappa_t\cup K_n\subset \kappa_s\cup F_s$;
\item[(c)] $U_s\subset U_t\cap O_{n+1}$;
\item[(d)] for every $x\in F_s$ the set $u_s(x)$ is contained in the set $S_{\mathsf N}(u_{s{\restriction}0}(x),\dots,u_{s{\restriction}n}(x))$ where $s{\restriction}i:=(K_0,\dots,K_{i-1})$ for $i\le n$;
\item[(e)] $[\kappa_s;U_s|F_s;u_s]\subset [K_n\setminus F_t;U_t\cap O_{n+1}]\cap [\kappa_t;U_t|F_t;u_t]\setminus M_{n+1}$.
\end{itemize}
Assume that for some $n\in\w$ and all sequences $s\in\K(X)^n$ the function $f_s$ and quadruple $(\kappa_s,U_s,F_s,u_s)\in\mathcal Q$ satisfying the inductive conditions (a)--(e) have been defined. Fix a sequence $s=(K_0,\dots,K_{n})\in\K(X)^{n+1}$ and consider the sequence $t=s{\restriction}n=(K_0,\dots,K_{n-1})$. For every $x\in\dot X$ consider the non-empty open set $$w_s(x)=S_{\mathsf N}(u_{s{\restriction}0},\dots,u_{s{\restriction}n}(x))\subset u_{s{\restriction}n}(x)=u_t(x).$$ The definition of $\mathcal Q$ ensures that for any $x\in \dot X\setminus F_t$ we have $u_{s{\restriction}i}(x)=Y$ for all $i\le n$ and then $w_s(x)=Y$ by the choice of the strategy $S_{\mathsf N}$. Then the quadruple $(\kappa_t,U_t,F_t,w_s)$ belongs to the family $\mathcal Q$.

It is easy to see that the open set $[K_{n}\setminus F_t;U_t\cap O_{n+1}]\cap [\kappa_t;U_t|F_t;w_s]\subset C_k'(X,Y)$ is not empty and hence contains some function $f_s\in C_k'(X,Y)$. Since the closed set $M_{n+1}$ is nowhere dense in $C_k'(X,Y)$, we can assume that $f_s\notin M_{n+1}$. Using Lemma~\ref{l:base}, find a quadruple $(\kappa_s,U_s,F_s,u_s)\in\mathcal Q$ such that $$f_s\in [\kappa_s;U_s|F_s;u_s]\subset [K_{n}\setminus F_t;U_t\cap O_{n+1}]\cap [\kappa_t;U_t|F_t;w_s]\setminus M_{n+1}$$and the conditions (a)--(e) are satisfied. 

Now define a strategy $S_{\mathsf F}:\K(X)^{<\w}\to\mathcal F(\dot X)$ of the player $\mathsf F$ in the game $\GKF$ letting $$S_{\mathsf F}(K_0,\dots,K_n)=F_{(K_0,\dots,K_n)}\setminus K_n=F_s\setminus K_n\mbox{ \ \ for \ }s=(K_0,\dots,K_n)\in\K(X)^{<\w}.$$ By Theorem~\ref{t:game}, the strategy $S_{\mathsf F}$ of $\mathsf F$ is not winning. So there exists an infinite sequence $s=(K_n)_{n\in\w}\in\K(X)^\w$ such that the family $\big\{S_{\mathsf F}(s{\restriction}n)\big\}_{n\in\w}$ is discrete in $X$. Then the set $\bigcup_{n\in\w}S_{\mathsf F}(s{\restriction}n)\subset\dot X$ is closed in $X$.

Consider the countable set $D:=\bigcup_{n\in\w}F_{s{\restriction}n}\subset\dot X$.  The inductive condition (e) ensures that $$u_{s{\restriction}n}(x)\subset S_{\mathsf N}(u_{s{\restriction}0}(x),\dots,u_{s{\restriction}(n-1)}(x))$$ for every $n\in\IN$. Since the strategy $S_{\mathsf N}$ is winning, the intersection $ \bigcap_{n\in\w}u_{s{\restriction}n}(x)$ is not empty and hence it contains some point $f_\infty(x)\in Y$.

We claim that the function $f:X\to Y$ defined by
$$f(x)=\begin{cases}
f_\infty(x)&\mbox{if $x\in D$};\\
*_Y&\mbox{otherwise};
\end{cases}
$$is continuous.

It suffices to check that $f$ is continuous at each non-isolated point $x\in X'$. Since $(O_n)_{n\in\w}$ is a neighborhood base at $*_Y$, for any $k\in\IN$ it suffices to find a neighborhood $W_x\subset X$ of $x$ such that $f(W_x)\subset O_k$.  Choose a neighborhood $W_x$ of $x$ which is disjoint with the closed set $$\bigcup_{i<k}F_{s{\restriction}i}\cup\bigcup_{i=k}^\infty S_{\mathsf F}(s{\restriction}i).$$
We claim that $f(w)\in O_k$ for all $w\in W_x$. This is clear if $w\notin D$. So assume that $w\in D$ and find the smallest number $m\in\w$ such that $w\in F_{s{\restriction}m}$. Then $w\notin F_{s{\restriction}(m-1)}$. The choice of the neighborhood $W_x$ ensures that $m\ge k$ and $$w\in F_{s{\restriction}m}\setminus S_{\mathsf F}(s{\restriction}m)=F_{s{\restriction}m}\setminus(F_{s{\restriction}m}\setminus K_{m-1})=F_{s{\restriction}m}\cap K_{m-1}.$$ Since $w\notin F_{s{\restriction}(m-1)}$, the inductive conditions (e) ensures that $$f(w)=f_\infty(w)\in u_{s{\restriction}m}(x)\subset O_{m}\subset O_k.$$ So $f$ is continuous and belongs to $C_k'(X,Y)$.

Next, we show that $f\in [\kappa_{s{\restriction}n};U_{s{\restriction}n}|F_{s{\restriction}n};u_{s{\restriction}n}]$ for every $n\in\IN$. Given any $x\in X$, we should prove that $f(x)\in U_{s{\restriction}n}$ if $x\in \kappa_{s{\restriction}n}$ and $f\in u_{s{\restriction}n}(x)$ if $x\in F_{s{\restriction}n}$. In the latter case the inclusion 
follows from $f(x)=f_\infty(x)\in u_{s{\restriction}n}(x)$. So, we assume that $x\in \kappa_{s{\restriction}n}$. If $x\notin D$, then $f(x)=*_Y\in U_{s{\restriction}n}$ and we are done. So, we assume that $x\in D$ and hence $x\in F_{s{\restriction}(m+1)}\setminus F_{s{\restriction}m}$ for some $m\in\w$. 
It follows from $x\in \kappa_{s{\restriction}n}\subset X\setminus F_{s{\restriction}n}$ that $m\ge n$.

The inductive conditions (b), (e) and (c) ensure that 
$x\in\kappa_{s{\restriction}n}\setminus F_{s{\restriction}m}\subset \kappa_{s{\restriction}m}$ and
$$f(x)=f_\infty(x)\in u_{s{\restriction}(m+1)}(x)\subset U_{s{\restriction}m}\subset U_{s{\restriction}n}.$$
Therefore, 
$$f\in  \bigcap_{n\in\w}[\kappa_{s{\restriction}n};U_{s{\restriction}n}|F_{s{\restriction}n};u_{s{\restriction}n}]\subset \bigcap_{n\in\w}(W\setminus M_n)=\emptyset$$and this is a desired contradiction, showing that the function space $C_k'(X,Y)$ is Baire.
\end{proof}

Lemmas~\ref{l:B1} and \ref{l:B2} imply the main result of this section.

\begin{theorem}\label{t:B} Assume that $(Y,*_Y)$ is a $*$-admissible $*$-first-countable Choquet pointed space. For any topological space $X$ the following conditions are equivalent:
\begin{enumerate}
\item[\textup{(1)}]  $C_k'(X,Y)$ is Baire;
\item[\textup{(2)}]  $C_k'(X,Y)$ is not meager;
\item[\textup{(3)}]  $X$ has $\DMOP$.
\end{enumerate}
\end{theorem}

For a topological space $X$ consider the function space
$$C_k'(X,2):=\{f\in C_k(X,2):f(X')\subset\{0\}\}\subset C_k(X,2)$$where the ordinal $2=\{0,1\}$ is endowed with the discrete topology.

\begin{definition} A topological space $X$ is defined to be
\begin{itemize}
\item {\em $C_k'$-meager} if the function space $C_k'(X,2)$ is meager;
\item {\em $C_k'$-Baire} if the function space $C_k'(X,2)$ is Baire;
\item {\em $C_k'$-Choquet} if the function space $C_k'(X,2)$ is Choquet.
\end{itemize}
\end{definition}

Theorem~\ref{t:B} implies the following function space characterization of $\DMOP$.

\begin{corollary}\label{c:B} For a topological space $X$ the following conditions are equivalent:
\begin{enumerate}
\item[\textup{(1)}] $X$ has $\DMOP$;
\item[\textup{(2)}] $X$ is $C_k'$-Baire;
\item[\textup{(3)}] $X$ is not $C_k'$-meager.
\end{enumerate}
\end{corollary}

\section{A function space characterization of $\WDMOP$}\label{s:W}

In this section we shall prove that a topological space has $\WDMOP$ if and only if it is $C_k'$-Choquet.

First, we prove a counterpart of Lemma~\ref{l:B1}.

\begin{lemma}\label{l:C1} Let $X$ be a topological space and $Y$ be a $*$-admissible pointed topological space. If the function space $C_k'(X,Y)$ is Choquet, then the space $X$ has $\WDMOP$.
\end{lemma}

\begin{proof} Assume that the function space $C_k'(X,Y)$ is Choquet, which means that the player $\mathsf N$ has a winning strategy $S_{\mathsf N}$ in the Choquet  game $\GEN(C_k'(X,Y))$.  The strategy $S_{\mathsf N}$ assigns to each finite sequence of non-empty open sets $(U_0,\dots,U_n)$ in $C_k'(X,Y)$ a non-empty open set $S_{\mathsf N}(U_0,\dots, U_n)\subset U_n$ so that for any infinite sequence $(U_n)_{n\in\w}$ of non-empty open sets in $C_k'(X,Y)$ the intersection $\bigcap_{n\in\w}S_{\mathsf N}(U_0,\dots,U_n)$ is not empty if $U_n\subset S_{\mathsf N}(U_0,\dots,U_{n-1})$ for every $n\in\IN$. Since the player $\mathsf E$ starts the game $\GEN(C_k'(X,Y))$, it convenient to assume that $S_{\mathsf N}(\emptyset)=C_k'(X,Y)$ for the empty sequence of zero length.

By our assumption, the pointed space $Y$ is $*$-admissible. So, there exists a non-empty open set $W\subset Y$ whose closure does not contain the distinguished point $*_Y$ of $Y$. By the definition of the compact-open topology, for any non-empty open set $U\subset C_k'(X,Y)$ there exists a compact set $\kappa(U)\subset X$ such that for any finite set $F\subset \dot X\setminus \kappa(U)$ the intersection $U\cap [F;W]$ is not empty.

Now we define a strategy $S_{\mathsf K}$ of the player $\mathsf K$ in the game $\GKF(X)$. This strategy assigns to any sequence $(F_0,\dots,F_{n-1})$ of finite sets in $\dot X$ the compact set $$K_n:=\kappa(U_{n})\cup\bigcup_{i<n}F_i$$ where $(U_0,\dots,U_{n})$ is a decreasing sequence of non-empty open sets in $C_k'(X,Y)$, defined  recursively by $U_0=C_k'(X,Y)$ and $$U_{i+1}=\begin{cases}S_{\mathsf N}(U_0\cap[F_0;W],\dots,U_i\cap[F_i;W]),&\mbox{if $U_i\cap [F_i;W]\ne\emptyset$};\\
U_i,&\mbox{otherwise};
\end{cases}\mbox{ \ \ for $i<n$}.
$$

We claim that this strategy $S_{\mathsf K}$ of  player $\mathsf K$ is winning.
Let $(F_n)_{n\in\w}$ be any sequence of finite subsets of $\dot X$ such that for any
$n\in\w$ the set $F_n$ is disjoint with the compact set $S_{\mathsf K}(F_0,\dots,F_{n-1})$. Let $U_0=C_k'(X,Y)$ and $(U_n)_{n\in\w}$ be the sequence of non-empty open sets in $C_k'(X,Y)$ defined by the recursive formula
$$U_{n+1}=S_{\mathsf N}(U_0\cap[F_0;W],\dots,U_n\cap [F_n;W])\subset U_n$$for $n\in\w$.
Let us show that for every $n\in\w$ the intersection $U_n\cap [F_n;W]$ is not empty, which means that the set $U_{n+1}$ is well-defined. The intersection $U_0\cap [F_0;W]$ is not empty as $U_0=C_k'(X,Y)$. Assume that for some $n\in\IN$ we have proved that the set $U_{n-1}\cap [F_{n-1};W]$ is not empty. Then the set $U_n=S_{\mathsf N}(U_0\cap[F_0;W],\dots,U_{n-1}\cap[F_{n-1};W])$ is not empty and by the definition of the strategy $S_{\mathsf K}$, we get $\kappa(U_{n})\subset S_{\mathsf K}(F_0,\dots,F_{n-1})$. Since $$F_{n}\cap \kappa(U_{n})\subset F_{n}\cap S_{\mathsf K}(F_0,\dots,F_{n-1})=\emptyset,$$ the definition of the compact set $\kappa(U_{n})$ ensures that the open set $U_n\cap [F_{n};W]$ is not empty.

Now observe that for the sequence of non-empty open sets $(U_n\cap[F_n;W])_{n\in\w}$ we have the inclusion
$$U_{n+1}\cap[F_{n+1};W]\subset U_{n+1}=S_{\mathsf N}(U_0\cap[F_0;W],\dots,U_n\cap [F_n;W])$$for all $n\in\w$. The winning property of the strategy $S_{\mathsf N}$ ensures that the intersection $\bigcap_{n\in\w}U_n\cap[F_n;W]$ is not empty and hence contains some  function $f\in C_k'(X,Y)$.

It follows that the family $(F_n)_{n\in\IN}$ is disjoint and $f\in \bigcap_{n\in\IN}[F_n;W]$. Since $f\big(\bigcup_{n\in\w}F_n\big)\subset W$ and $*_Y\notin \overline{W}$, the continuity of $f$ guarantees that the closure of the set $\bigcup_{n=1}^\infty F_n$ does not intersect $X'$ and hence the disjoint family $\{F_n\}_{n\in\IN}$ is discrete in $X$.

For every $n\in\w$ let $K_n=S_{\mathsf K}(F_0,\dots,F_{n-1})$ and observe that  $$K_1,F_1,K_2,F_2,\dots$$is the sequence of the moves of the players $\mathsf K$ and $\mathsf F$ in the game $\GKF(X)$, where the player $\mathsf K$ plays according to the strategy $S_{\mathsf K}$ and eventually wins as the family $\{F_n\}_{n\in\IN}$ is discrete. So, the strategy $S_{\mathsf K}$ is  winning and the space $X$ has $\WDMOP$.
\end{proof}

Now we prove a ``Choquet'' version of Lemma~\ref{l:B2}.

\begin{lemma}\label{l:C2} Let $Y$ be a $*$-first-countable pointed Choquet space. If a topological  space $X$ has $\WDMOP$, then the function space $C_k'(X,Y)$ is Choquet.
\end{lemma}

\begin{proof}
By our assumption, the distinguished point $*_Y$ of $Y$ has a countable neighborhood base $\{O_n\}_{n\in\w}$ such that $O_{n+1}\subset O_n$ for all $n\in\w$. Let $\tau_Y$ be the family of all non-empty open sets in $Y$. The family $\tau_Y$ plays the role of the base $\mathcal B_Y$ in the definition of the family $\mathcal Q$ of quadruples from Section~\ref{s3}.

Since the space $Y$ is Choquet, the player $\mathsf N$ has a winning strategy $S_{\mathsf N}:\tau_Y^{<\w}\to\tau_Y$ in the Choquet game $\GEN(Y)$. The strategy $S_{\mathsf N}$ is a function assigning to each finite sequence $(U_0,\dots,U_n)\in\tau_Y^{<\w}$ of non-empty open sets in $Y$ a non-empty open set $S_{\mathsf N}(U_0,\dots,U_n)\subset U_n$ such that for any infinite sequence $(U_n)_{n\in\w}\in\tau_Y^\w$ the intersection $\bigcap_{n\in\w}S_{\mathsf N}(U_0,\dots,U_n)$ is not empty if $U_n\subset S_{\mathsf N}(U_0,\dots,U_{n-1})$ for any $n\in\IN$. Since the player $\mathsf E$ starts the game $\GEN(Y)$, it is convenient to assume that $S_{\mathsf N}(\emptyset)=Y$ for the empty sequence $\emptyset\in \tau_Y^0$. We can also assume that $S_{\mathsf N}(U_0,\dots,U_{n-1})=Y$ if $U_i=Y$ for all $i<n$.

By Lemma~\ref{l:sW}, the player $\mathsf K$ has a strategy $S_{\mathsf K}$ in the game $\GKF$ such that for any sequence $(F_n)_{n\in\w}$ of finite subsets of $\dot X$ the indexed family $\{F_n\setminus S_{\mathsf K}(F_0,\dots,F_{n-1})\}_{n\in\w}$ is discrete in $X$.

Let $\tau$ be the family of all non-empty open sets in $C_k'(X,Y)$. Let $\kappa_\emptyset=F_\emptyset=\emptyset$, $U_\emptyset=Y$ and $u_\emptyset:\dot X\to\{Y\}\subset\tau_Y$ be the constant function.

For any $n\in\w$ and any finite sequence $s=(W_0,\dots,W_n)\in\tau^{n+1}$ we shall define inductively a compact set $K_s\subset X$ and two quadruples $(\tilde \kappa_s, \tilde U_s,\tilde F_s,\tilde u_s)$, $(\kappa_s, U_s,F_s,u_s)$ in $\mathcal Q$ such that  the following conditions are satisfied:
\begin{itemize}
\item[(a)] $[\tilde \kappa_s;\tilde U_s|\tilde F_s;\tilde u_s]\subset W_n$;
\item[(b)] $\tilde U_s\subset O_{n+1}\cap U_{s{\restriction}n}$, $F_{s{\restriction}n}\subset\tilde F_s$, and $\kappa_{s{\restriction}n}\subset\tilde\kappa_s\cup\tilde F_s$;
\item[(c)] $K_s=S_{\mathsf K}(\tilde F_{s{\restriction}0},\dots,\tilde F_{s{\restriction}(n+1)})$;
\item[(d)] $F_s=\tilde F_s$, $U_s=\tilde U_s$ and $\kappa_s=\tilde \kappa_s\cup K_s\setminus \tilde F_s$;
\item[(e)]  $u_s(x)=S_{\mathsf N}(\tilde u_{s{\restriction}0}(x),\dots,\tilde u_{s{\restriction}(n+1)}(x))\subset\tilde u_{s{\restriction}(n+1)}(x)=\tilde u_s(x)$ for every $x\in \dot X$;
\item[(f)] $[\kappa_s;U_s|F_s;u_s]\subset [\tilde \kappa_s;\tilde U_s|\tilde F_s;\tilde u_s]\subset W_n$.
\end{itemize}

Assume that for some $n\in\w$ and all sequences $s\in\tau^n$ a compact set $K_s$, and quadruples $(\tilde\kappa_s,\tilde U_s,\tilde F_s,\tilde u_s)$, $(\kappa_s,U_s,F_s,u_s)$  satisfying the inductive conditions (a)--(f) have been defined. Fix a sequence $s=(W_0,\dots,W_{n})\in\tau^{n+1}$.

Using Lemma~\ref{l:base}, find a quadruple $(\tilde\kappa_s,\tilde U_s,\tilde F_s,\tilde u_s)\in\mathcal Q$ satisfying the conditions (a), (b). Define the compact set $K_s$ by the formula (c). Finally define the quadruple $(\kappa_s,U_s,F_s,u_s)$ by the conditions (d) and (e). The conditions (a), (d), (e) imply the condition (f).
This completes the inductive step.

After completing the inductive construction, define a strategy $\$_{\mathsf N}:\tau^{<\w}\to\tau$ of the player $\mathsf N$ in the Choquet game $\GEN(C_k'(X,Y))$ letting $\$_{\mathsf N}(s)=[\kappa_s;U_s|F_s;u_s]$ for any sequence $s=(W_0,\dots,W_{n})\in\tau^{<\w}$. The inductive condition (f) guarantees that $\$_{\mathsf N}(W_0,\dots,W_{n})\subset W_n$.


We claim that the strategy $\$_{\mathsf N}$ is winning. Fix an infinite sequence $s=(W_n)_{n\in\w}\in\tau^\w$ such that $W_n\subset \$_{\mathsf N}(W_0,\dots,W_{n-1})=\$_{\mathsf N}(s{\restriction}n)$ for every $n\in\IN$. The condition (f)  of the inductive construction ensures that for every $n\in\IN$ we have the inclusions
$$
\begin{aligned}
&[\kappa_{s{\restriction}n};U_{s{\restriction}n}|F_{s{\restriction}n};u_{s{\restriction}n}]\subset [\tilde \kappa_{s{\restriction}n};\tilde U_{s{\restriction}n}|\tilde F_{s{\restriction}n};\tilde u_{s{\restriction}n}]\subset W_{n-1}\subset \\
&\subset\$_{\mathsf N}(s{\restriction}(n-1))=[\kappa_{s{\restriction}(n-1)};U_{s{\restriction}(n-1)}|F_{s{\restriction}(n-1)};u_{s{\restriction}(n-1)}],
\end{aligned}
$$
which imply the inclusions $$u_{s{\restriction}n}(x)\subset \tilde u_{s{\restriction}n}(x)\subset u_{s{\restriction}(n-1)}(x)=S_{\mathsf N}(\tilde u_{s{\restriction}0}(x),\dots,\tilde u_{s{\restriction}(n-1)}(x))$$holding for every $x\in\dot X$.

Since the strategy $S_{\mathsf N}$ is winning, for every $x\in\dot X$ the intersection
$$\bigcap_{n\in\w}\tilde u_{s{\restriction}n}(x)=  \bigcap_{n\in\w}u_{s{\restriction}n}(x)$$is not empty and hence contains some point $f_\infty(x)\in Y$.

The choice of the strategy $S_{\mathsf K}$ guarantees that the indexed family $$\{F_{s{\restriction}n}\setminus S_{\mathsf K}(F_{s{\restriction}0},\dots,F_{s{\restriction}(n-1)})\}_{n\in\w}$$ is discrete in $X$.

Consider the countable set $D=\bigcup_{n\in\w}F_{s{\restriction}n}\subset\dot X$.
We claim that the function $f:X\to Y$ defined by
$$f(x)=\begin{cases}
f_\infty(x)&\mbox{if $x\in D$}\\
0&\mbox{otherwise}
\end{cases}
$$is continuous.

It suffices to check that $f$ is continuous at each non-isolated point $x\in X'$. Since $(O_n)_{n\in\w}$ is a neighborhood base at $*_Y$, for any $k\in\IN$ it suffices to find a neighborhood $W_x\subset X$ of $x$ such that $f(W_x)\subset O_k$.  Choose a neighborhood $W_x$ of $x$ which is disjoint with the set $$\bigcup_{i<  k}F_{s{\restriction}(i+1)}\cup\bigcup_{i=k}^\infty (F_{s{\restriction}(i+1)}\setminus S_{\mathsf K}(F_{s{\restriction}0},\dots,F_{s{\restriction}i})).$$
We claim that $f(z)\in O_k$ for all $w\in W_x$. This is clear if $w\notin D$. So assume that $w\in D$ and find the smallest number $m\in\w$ such that $w\in F_{s{\restriction}(m+1)}$. Then $w\notin F_{s{\restriction}m}$. The choice of the neighborhood $W_x$ ensures that $m\ge k$ and $$w\in  S_{\mathsf K}(F_{s{\restriction}0},\dots,F_{s{\restriction}m})\setminus F_{s{\restriction}m}=K_{s{\restriction}m}\setminus F_{s{\restriction}m}\subset \kappa_{s{\restriction}m}.$$ The inductive conditions (b), (d) and the inclusion
$$
 [\kappa_{s{\restriction}(m+1)};U_{s{\restriction}(m+1)}|F_{s{\restriction}(m+1)};u_{s{\restriction}(m+1)}]\subset [\kappa_{s{\restriction}m};U_{s{\restriction}m}|F_{s{\restriction}m};u_{s{\restriction}m}]
$$
ensure that $f(w)=f_\infty(w)\in u_{s{\restriction}(m+1)}(w)\subset U_{s{\restriction}m}\subset O_{m}\subset O_k$. So $f$ is continuous and belongs to $C_k'(X,Y)$.

Repeating the argument from the proof of Lemma~\ref{l:B2}, we can show that 
$$f\in  \bigcap_{n\in\IN}[\kappa_{s{\restriction}n};U_{s{\restriction}n}|F_{s{\restriction}n};u_{s{\restriction}n}]=\bigcap_{n\in\IN}W_{n-1},$$so the intersection $\bigcap_{n\in\w}W_n$ is not empty and the strategy $\$_{\mathsf N}$ is winning, which means that the function space $C_k'(X,Y)$ is Choquet.
\end{proof}

Lemmas~\ref{l:C1} and \ref{l:C2} imply the main result of this section.

\begin{theorem}\label{t:C} Assume that $Y$ is a $*$-admissible $*$-first-countable pointed Choquet space. For any topological space $X$ the following conditions are equivalent:
\begin{enumerate}
\item[\textup{(1)}]  $C_k'(X,Y)$ is Choquet;
\item[\textup{(2)}]  $X$ has $\WDMOP$;
\item[\textup{(3)}]  $X$ is $C_k'$-Choquet.
\end{enumerate}
\end{theorem}

\section{Introducing properties of $X$, responsible for the metrizability of  $C'_k(X,Y)$}\label{s:kappa}


In this section we introduce some properties of a topological space $X$, which will be used in subsequent sections for characterizing pairs $X,Y$ for which the function spaces $C_k'(X,Y)$ are metrizable, (almost) complete-metrizable or (almost) Polish.

\begin{definition}\label{d:kappa} A topological space $X$ is called
\begin{itemize}
\item {\em $\dot\kappa$-space} if for any non-closed  subset $D\subset \dot X$ of $X$ there exists a compact set $K\subset X$ such that $K\cap D$ is infinite;
\item {\em $\w$-$\dot\kappa$-space} if for any non-closed countable subset $D\subset \dot X$ of $X$ there exists a compact set $K\subset X$ such that $K\cap D$ is infinite;
\item {\em $\dot \kappa_\w$-space} if there exists a countable family $\{K_n\}_{n\in\w}$ of compact subsets of $X$ such that $\dot X\subset \bigcup_{n\in\w}K_n$ and for any non-closed set $D\subset \dot X$ of $X$ there exists $n\in\w$ such that $D\cap K_n$ is infinite;
\item {\em hemi-$\dot\kappa_\w$-space} if there exists a countable family $\{K_n\}_{n\in\w}$ of compact subsets of $X$ such that for any compact set $K\subset X$ there exists $n\in\w$ such that $K\cap \dot X\subset K_n$.
\end{itemize}
\end{definition}

These properties relate as follows
$$
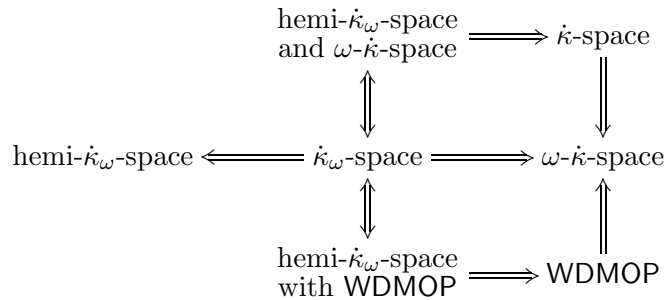
\xymatrix{
&\mbox{hemi-$\dot\kappa_\w$-space}\atop{\mbox{and $\w$-$\dot\kappa$-space}}\ar@{<=>}[d]\ar@{=>}[r]&\mbox{$\dot\kappa$-space}\ar@{=>}[d]\\
\mbox{hemi-$\dot\kappa_\w$-space}&\mbox{$\dot\kappa_\w$-space}\ar@{=>}[l]\ar@{<=>}[d]\ar@{=>}[r]&\w\mbox{-$\dot\kappa$-space}\\
&\mbox{hemi-$\dot\kappa_\w$-space}\atop\mbox{with $\WDMOP$}\ar@{=>}[r]&\WDMOP\ar@{=>}[u]
}
$$

Non-trivial implications in this diagram are proved in the following proposition.

\begin{proposition}\label{p:kappa}
\begin{enumerate}
\item[\textup{(1)}] Each $\dot\kappa_\w$-space is both a $\dot \kappa$-space and a hemi-$\dot\kappa_\w$-space.
\item[\textup{(2)}] A topological space is a $\dot\kappa_\w$-space if and only if it is  a hemi-$\dot\kappa_\w$-space and an $\w$-$\dot\kappa$-space.
\item[\textup{(3)}] Each $\dot\kappa_\w$-space has $\WDMOP$.
\item[\textup{(4)}] Each space with $\WDMOP$ is an $\w$-$\dot\kappa$-space.
\item[\textup{(5)}] A topological space is a $\dot\kappa_\w$-space if and only if it is a hemi-$\dot\kappa_\w$-space with $\WDMOP$.
\end{enumerate}
\end{proposition}

\begin{proof} 1. The definitions imply that each $\dot\kappa_\w$-space $X$ is a $\dot\kappa$-space. To show that $X$ is a hemi-$\dot\kappa_\w$-space, take any sequence $(K_n)_{n\in\w}$ of compact subsets of $X$, witnessing that $X$ is a $\dot\kappa_\w$-space. The hemi-$\dot\kappa_\w$-space property of $X$ will follow as soon as we check that for any compact set $K\subset X$ there exists $n\in\w$ such that $K\cap\dot X\subset \bigcup_{i\le n}K_i$.
Assuming no such $n$ exists, for every $n\in\w$ we can choose a point $x_n\in K\cap\dot X\setminus \bigcup_{i\le n}K_i$. Since $\dot X\subset\bigcup_{n\in\w}K_n$, the set  $D=\{x_n\}_{n\in\w}\subset K$ is infinite and hence has an accumulation point $x'\in K\cap X'\subset K\setminus D$ in the compact space $K$. Therefore, $D$ is not closed in $X$ and the choice of the sequence $(K_n)_{n\in\w}$ yields a number $n\in\w$ such that  $D\cap K_n$ is infinite. On the other hand, $D\cap K_n\subset\{x_0,\dots,x_{n-1}\}$ by the choice of the sequence $(x_n)_{n\in\w}$. This contradiction shows that $X$ is a hemi-$\dot\kappa_\w$-space.
\smallskip

2. Assume that $X$ is a hemi-$\dot\kappa_\w$-space and an $\w$-$\dot\kappa$-space. Let $(K_n)_{n\in\w}$ be a sequence witnessing that $X$ is a hemi-$\dot\kappa_\w$-space. We claim that this sequence witness that $X$ is $\dot\kappa_\w$-space. Since each singleton $\{x\}\subset\dot X$ is compact, there exists $n\in\w$ such that $\{x\}=\{x\}\cap\dot X\subset K_n$, which implies that $\dot X\subset\bigcup_{n\in\w}K_n$. Given a non-closed subset $D\subset \dot X$, it remains to find $n\in\w$ such that $K_n\cap D$ is infinite. To derive a contradiction, assume that $D\cap K_n$ is finite for every $n\in\w$. Then the set $D=D\cap\dot X=\bigcup_{n\in\w}D\cap K_n$ is countable. Since $X$ is an $\w$-$\dot\kappa$-space, there exists a compact set $K\subset X$ such that $K\cap D$ is infinite. The choice of the sequence $(K_n)_{n\in\w}$ ensures that $\dot X\cap K\subset K_n$ for some $n\in\w$. Since $D\cap K=D\cap\dot X\cap K\subset D\cap K_n$, the set $D\cap K_n$ is infinite, which contradicts our assumption. This contradiction completes the proof.
\smallskip


3. Assume that $X$ is a $\dot\kappa_\w$-space and let $(K_n)_{n\in\w}$ be a sequence of compact sets witnessing this fact. It is easy to see that the function $S_{\mathsf K}$ assigning to any finite sequence $(F_0,\dots,F_n)\in\mathcal F(\dot X)^{<\w}$ the compact set $$S_{\mathsf K}(F_0,\dots,F_n)=\bigcup_{i\le n}(F_i\cup K_i)$$ is a winning strategy of the player $\mathsf K$ in the game $\GKF(X)$.
\smallskip

4. Assume that $X$ has $\WDMOP$ and let $A\subset \dot X$ be a non-closed countable set in $X$. To derive a contradiction, assume that every compact set $K\subset X$ the intersection $A\cap K$ is finite.

By Theorem~\ref{t:C}, the function space $C_k'(X,2)$ is Choquet. The space $C_k'(X,2)$ is a subspace of the space $F'_k(X,2)=\{f\in 2^X:f(X')\subset\{0\}\}$, endowed with the compact-open topology. A neighborhood base of this topology at a function $f\in F_k'(X,2)$ consists of the sets $O_K(f)=\{g\in F_k'(X,2):g{\restriction}K=f{\restriction}K\}$ where $K$ runs over compact subsets of $X$.

Endow the doubleton $2=\{0,1\}$ with a group operation $\oplus$ in which $0$ is a neutral element. This group operation induces a continuous group operation on the space $F_k(X,2)$.

Let $\chi\in F_k'(X,2)$ be the function defined by $\chi^{-1}(1)=A$. Since the set $A$ is not closed in $X$, the function $\chi$ is discontinuous. On the other hand, this function belongs to the closure $\overline{C_k'(X,2)}$ of the subgroup $C_k'(X,2)$ in $F_k'(X,2)$ as for any compact set $K\subset X$ the intersection $K\cap A\subset\dot X$ is finite and hence the restriction $\chi{\restriction}K$ is continuous.

By \cite[Theorem 3]{BH}, each Choquet topological group $H$ is $G_\delta$-dense in its Raikov completion $\bar H$, which means that $H$ has non-empty intersection with any non-empty $G_\delta$-subset of $\bar H$.

This fact implies that the Choquet subgroup $C_k'(X,2)$ is $G_\delta$-dense in its closure $\overline{C_k'(X,2)}$. Then the $G_\delta$-set $G=\{f\in \overline{C_k'(X,2)}:f{\restriction}A=\chi{\restriction}A\}$ has common point $f$ with the subgroup $C_k'(X,2)$.  Now get a desired contradiction:
$$\emptyset\ne X'\cap \bar A\subset f^{-1}(0)\cap \overline{f^{-1}(1)}=f^{-1}(0)\cap f^{-1}(1)=\emptyset.$$
\vskip5pt

5. If $X$ is a $\dot\kappa_\w$-space, then it is a hemi-$\dot\kappa_\w$-spaces with $\WDMOP$ by the statements (1) and (3), proved above.
If $X$ is a hemi-$\dot\kappa_\w$-space with $\WDMOP$, then $X$ is a $\dot\kappa_\w$-space by the statements (2) and (4), proved above.
\end{proof}

For topological spaces with countable set of isolated points we can prove a bit more.

\begin{theorem}\label{t:W=k} Let $X$ be a topological space with countable set $\dot X$ of isolated points. Then
\begin{enumerate}
\item[\textup{(1)}] $X$ is a $\dot \kappa$-space if and only if it is a $\w$-$\dot\kappa$-space;
\item[\textup{(2)}] $X$ is a $\dot \kappa_\omega$-space if and only if $X$ has $\WDMOP$.
\end{enumerate}
\end{theorem}

\begin{proof} 1. The first equivalent follows directly from the definitions.
\smallskip

2. If $X$ is a $\dot\kappa_\w$-space, then it has $\WDMOP$ by Proposition~\ref{p:kappa}(3). Now assume that $X$ has $\WDMOP$ and fix a winning strategy $S_{\mathsf N}:\mathcal F(\dot X)^{<\w}\to\K(X)$ of the player $\mathsf K$ in the game $\GKF(X)$. Here $\F(\dot X)$ is the family of finite subsets of $\dot X$. Since the set $\dot X$ is countable, the family $\mathcal K=\F(\dot X)\cup \{S_{\mathsf K}(s):s\in\F(\dot X)^{<\w}\}$ is countable, too.

Let $\K=\{K_n\}_{n\in\w}$ be an enumeration of the family $\K$. We claim that the sequence $(K_{\le n})_{n\in\w}$ of the compact sets $K_{\le n}=\bigcup_{i\le n}K_i$ witnesses that $X$ is a $\dot\kappa_\w$-space.

Given a non-closed subset $D\subset\dot X$, we should show that the intersection $D\cap K_{\le n}$ is infinite for some $n\in\w$. By Proposition~\ref{p:kappa}(4), there is a compact set $K\subset X$ such that the intersection $K\cap D$ is infinite.
We claim that there exists $n\in\w$ such that $K\cap\dot X\subset K_{\le n}$. To derive a contradiction, assume that for every $n\in\w$ the complement $K\cap\dot X\setminus K_{\le n}$ contains some point $x_n\in\dot X$. Since each subsequence of the sequence $(x_n)_{n\in\w}$ has an accumulation point in the compact space $K\subset X$, the player $\mathsf F$ has a winning strategy by choosing his moves in the set $\{x_n\}_{n\in\w}$. But this is not possible as player $\mathsf K$ has a winning strategy in the game $\GKF(X)$. This contradiction shows that $K\cap D\subset K\cap \dot X\subset K_{\le n}$ for some $n\in\w$. Then the intersection $D\cap K_{\le n}\supset D\cap K$ is infinite.
\end{proof}

Theorem~\ref{t:W=k} implies that for any topological space $X$ with countable set $\dot X$ we have the following equivalences and implications:
$${\mbox{hemi-$\dot\kappa_\w$-space}\atop{\mbox{and $\w$-$\dot\kappa$-space}}}\;\Leftrightarrow\; \mbox{$\dot\kappa_\w$-space} \;\Leftrightarrow\; \WDMOP  \;\Ra \mbox{$\dot\kappa$-space}\;\Leftrightarrow\;\mbox{$\w$-$\dot\kappa$-space}.$$

\section{The metrizability of the function spaces $C_k'(X,Y)$}\label{s:M}

\begin{lemma}\label{l:M1} Let $Y$ be a pointed topological space whose distinguished point $*_Y$ has a neighborhood  $U_*\ne Y$. A topological space $X$ is a hemi-$\dot\kappa_\w$-space if the function space $C_k'(X,Y)$ is first-countable at the constant function $c:X\to\{*_Y\}$.
\end{lemma}

\begin{proof} Being first-countable at $c$, the function space $C_k'(X,Y)$ has a countable neighborhood base $\{O_n\}_{n\in\w}$ at $c$. By the definition of the compact-open topology on $C_k'(X,Y)$, for every $n\in\w$ we can find a compact subset $K_n\subset X$ and a neighborhood $V_n\subset Y$ of $*_Y$ such that $c\in [K_n;V_n]\subset O_n$. We claim that the sequence $(K_n)_{n\in\w}$ witnesses that $X$ is a hemi-$\dot\kappa_\w$-space.

Given a compact set $K\subset X$, we should find $n\in\w$ such that $K\cap\dot X\subset K_n$. By our assumption, the distinguished point $*_Y$ has an open neighborhood $U_*\ne Y$. Consider the open neighborhood $[K,U_*]\subset C_k'(X,Y)$ of the constant function $c$ and find $n\in\w$ with $O_n\subset [K;U_*]$. Then the inclusion $[K_n;V_n]\subset O_n\subset[K;U_*]$ implies $K\cap \dot X\subset K_n$. Indeed, assuming that $K\cap\dot X\not\subset K_n$, we can find an isolated point $x\in K\cap\dot X\setminus K_n$.  Fix any point $y\in Y\setminus U_*$ and consider the map $\chi_x:X\to\{*_Y,y\}\subset Y$ defined by $\chi_x^{-1}(y)=x$. Since $x$ is isolated in $X$, the map $\chi_x$ is continuous. Taking into account that $\chi_x(K_n)\subset\{*_Y\}$ and $\chi_x(x)=y\notin U_*$, we conclude that   $\chi_x\in [K_n;V_n]\setminus [K;U_*]$, which contradicts the choice of $n$ (with $[K_n;V_n]\subset O_n\subset[K;U_*]$).
\end{proof}

For the proof of Corollary~\ref{c:AM}, we shall need the following refined version of Lemma~\ref{l:M1}.

\begin{lemma}\label{l:M1a} Let $Y$ be a pointed topological space such that each point $y\in Y$ has a neighborhood $O_y$ which is not dense in $Y$. A topological space $X$ is a hemi-$\dot\kappa_\w$-space if the function space $C_k'(X,Y)$ contains a dense subspace $D$, which is first-countable at some point $\delta\in D$.
\end{lemma}

\begin{proof} Let $\{O_n\}_{n\in\w}$ be a countable neighborhood base at the point $\delta$ of the space $D$. By Lemma~\ref{l:base}, for every $n\in\w$ we can find a quadruple $(K_n,U_n,F_n,u_n)\in\mathcal Q$ such that $\delta\in D\cap [K_n;U_n|F_n;u_n]\subset O_n$. We claim that the sequence $(K_n\cup F_n)_{n\in\w}$ witnesses that $X$ is a hemi-$\dot\kappa_\w$-space.

Given a compact set $K\subset X$, we should find $n\in\w$ such that $K\cap\dot X\subset K_n\cup F_n$. By our assumption, each point $y\in Y$ has a neighborhood $O_y$, which is not dense in $Y$. By Lemma~\ref{l:base}, there exists a quadruple $(\kappa,V,E,v)\in\mathcal Q$ such that $\delta\in[\kappa;V|E;v]$, $K\subset \kappa\cup E$, $V\subset O_{*_Y}$ and $v(x)\subset O_{\delta(x)}$ for all $x\in E$.
Since $\{O_n\}_{n\in\w}$ is a neighborhood base of the space $D$ at $\delta$, there exists $n\in\w$ such that $O_n\subset [\kappa;V|E;v]$.

We claim that $K\cap\dot X\subset K_n\cup F_n$. Assuming that $K\cap\dot X\not\subset K_n\cup F_n$, we can find a point $x\in \dot X\cap K\setminus (K_n\cup F_n)$.
Choose a non-empty open set $W_x\subset Y$ such that $W_x\cap O_{\delta(x)}=\emptyset$ if $x\in E$ and $W_x\cap V=\emptyset$ if $x\notin E$.
By the density of the set $D$ in $C_k'(X,Y)$, the intersection $D\cap [K_n;U_n|F_n;u_n]\cap[\{x\};W_x]$ contains some function $f$. Then $f\in D\cap [K_n;U_n|F_n;u_n]\subset O_n\subset [\kappa;V|E;v]$. Since $x\in K\subset \kappa\cup E$, the inclusion $f\in [\kappa;V|E;v]$ implies that $f(x)\in V$ if $x\in\kappa$ and $f(x)\in v(x)\subset O_{\delta(x)}$ if $x\in E$. In both cases, we get $f(x)\notin W_x$, which contradicts the inclusion $f\in[\{x\};W_x]$. This contradiction shows that $K\cap\dot X\subset K_n\cup F_n$. Therefore, the sequence of compact sets $(K_n\cup F_n)_{n\in\w}$   witnesses that $X$ is a hemi-$\dot\kappa_\w$-space.
\end{proof}

\begin{lemma}\label{l:M2} For any hemi-$\dot\kappa_\w$-space $X$ and any metrizable space $Y$, the function space $C_k'(X,Y)$ is metrizable.
\end{lemma}

\begin{proof} Let $(K_n)_{n\in\w}$ be a sequence of compact sets witnessing that $X$ is a hemi-$\dot\kappa_\w$-space. Replacing each compact set $K_n$ by the union $\bigcup_{i\le n}K_i$, we can assume that $K_n\subset K_{n+1}$ for all $n\in\w$.

Let $d$ be any metric generating the topology of the space $Y$. For a point $y\in Y$ and $\e>0$ denote by $B(y;\e):=\{x\in Y:d(y,x)<\e\}$ the open $\e$-ball centered at $y$.

We claim that the metric $\rho$ on $C_k'(X,Y)$ defined by the formula
$$\rho(f,g)=\max_{n\in\w}\min\{\tfrac1{2^n},\max_{x\in K_n}d(f(x),g(x))\}\quad\mbox{for $f,g\in C_k'(X,Y)$}$$
generates the compact-open topology of the function space $C_k'(X,Y)$.

By \cite[8.2.7]{Eng}, the compact-open topology on $C_k'(X,Y)$ coincides with the topology of uniform convergence on compacta, which implies that the metric $\rho$ is continuous. It remains to check that for any compact set $K\subset X$ and any open set $U\subset Y$ the subbasic open set $[K;U]:=\{f\in C_k'(X,Y):f(K)\subset U\}$ is open in the metric space $(C_k'(X,Y),\rho)$. Fix any function $f\in[K;U]$.
By the choice of the sequence $(K_n)_{n\in\w}$, there exists a number $n\in\w$ such that $K\cap\dot X\subset K_n$.
By the compactness of the set $f(K)\subset U$, there exists a positive real number $\e<\frac1{2^n}$ such that $B(f(K);\e)\subset U$, where $B(f(K);\e)=\bigcup_{y\in f(K)}B(y;\e)$.

 We claim that any function $g\in C_k'(X,Y)$ with $\rho(f,g)<\e$ belongs to $[K;U]$. Given any $z\in K$, we should check that $g(z)\in U$. If $z\in X'$, then $g(z)=*_Y=f(z)\in f(K)\subset U$. If $z\notin X'$, then $z\in K\cap\dot X\subset K_n$ and then $\e<\frac1{2^n}$ and
 $$\min\big\{\tfrac1{2^n},\max_{x\in K_n}d(f(x),g(y))\big\}\le\rho(f,g)<\e$$ imply that $d(g(z),f(z))\le \max_{x\in K_n}d(f(x),g(x))<\e$ and finally, $g(z)\in B(f(z);\e)\subset B(f(K),\e)\subset U$.
\end{proof}

\begin{theorem}\label{t:M} For a pointed topological space $Y$ containing more than one point, and a topological space $X$ containing an  isolated point, the following conditions are equivalent:
\begin{enumerate}
\item[\textup{(1)}] the function space $C_k'(X,Y)$ is metrizable;
\item[\textup{(2)}] the space $Y$ is metrizable and $X$ is a hemi-$\dot\kappa_\w$-space.
\end{enumerate}
\end{theorem}

\begin{proof} The implication $(2)\Ra(1)$ was proved in Lemma~\ref{l:M2}.

To prove that $(1)\Ra(2)$, assume that the function space $C_k'(X,Y)$ is metrizable. By Lemma~\ref{l:M1}, $X$ is a hemi-$\dot\kappa_\w$-space. By our assumption, the space $X$ contains an isolated point $x\in X$. It is easy to see that the function
$$H:C_k'(X,Y)\to Y\times C_k'(X\setminus\{x\},Y),\;\;f\mapsto (f(x),f{\restriction}X\setminus\{x\}),$$
is a homeomorphism. Then the space $Y$ is metrizable, being homeomorphic to a subspace the metrizable space $Y\times C_k'(X\setminus\{x\},Y)=H(C_k'(X,Y))$.
\end{proof}

Theorem~\ref{t:M} and Lemma~\ref{l:M1a} imply the following corollary.

\begin{corollary}\label{c:AM} For a pointed metrizable space $Y$ containing more than one point, and a topological space $X$ containing an  isolated point, the following conditions are equivalent:
\begin{enumerate}
\item[\textup{(1)}] the function space $C_k'(X,Y)$ is metrizable;
\item[\textup{(2)}] the function space $C'_k(X,Y)$ contains a dense first-countable subspace;
\item[\textup{(3)}] $X$ is a hemi-$\dot\kappa_\w$-space.
\end{enumerate}
\end{corollary}

A topological space $X$ is defined to have {\em countable spread} if each discrete subspace in $X$ is at most countable.

\begin{lemma}\label{l:spread} Let $Y$ be a pointed space such that the singleton $\{*_Y\}$ is not dense in $Y$. A topological space $X$ has countable set $\dot X$ of isolated points if the function space $C_k'(X,Y)$ has countable spread.
\end{lemma}

\begin{proof} By our assumption, the space $Y$ contains a point $y\in Y\setminus\overline{\{*_Y\}}$. For every $x\in\dot X$ let $\delta_x:X\to\{*_Y,y\}$ be the unique function with $\delta_x^{-1}(y)=\{x\}$. Observe that the set $U_x:=\{f\in C_k'(X,Y):f(x)\notin \overline{\{*_Y\}}$ is an open neighborhood of $\delta_x$ in $C_k'(X,Y)$. Consider the subspace $D=\{\delta_x:x\in\dot X\}$ and observe that it is discrete, because $U_x\cap D=\{\delta_x\}$ for all $x\in \dot X$. Since the space $C_k'(X,Y)$ has countable spread, the discrete subspace $D$ has at most countable cardinality $\w\ge|D|=|\dot X|$.
\end{proof}

\begin{lemma}\label{l:MS2} Let $Y$ be a $*$-admissible pointed space and $X$ be a hemi-$\dot\kappa_\w$-space. If the function space $C_k'(X,Y)$ has countable cellularity, then $\dot X$ is at most countable.
\end{lemma}

\begin{proof} Assuming that $\dot X$ is uncountable and taking into account that $X$ is a hemi-$\dot\kappa_\w$-space, we can find a compact set $K\subset X$ such that $\dot X\cap K$ is uncountable.

Since the pointed space $Y$ is $*$-admissible, the point $*_Y$ has a neighborhood $U_*$ which is disjoint with some non-empty open set $V\subset Y$. For every $x\in K\cap\dot X$ consider the open set
$$W_x:=[\{x\};V]\cap[K\setminus\{x\};U_*]$$in $C_k'(X,Y)$ and observe that $(W_x)_{x\in \dot X\cap K}$ is an uncountable family of pairwise disjoint open sets in $C_k'(X,Y)$, which means that $C_k'(X,Y)$ has uncountable cellularity.
\end{proof}

\begin{lemma}\label{l:separable} For any separable pointed space $Y$ and any topological space $X$ with countable set $\dot X$ of isolated point, the function space $C_k'(X,Y)$ is separable.
\end{lemma}

\begin{proof} Fix a countable dense set $D\subset Y$, containing the distinguished point $*_Y$ of $Y$. It is easy to see that the countable set $$\{f\in C_k'(X,Y):f(X)\subset D,\;\;|X\setminus f^{-1}(*_Y)|<\w\}$$is dense in $C_k'(X,Y)$, which means that the function space $C_k'(X,Y)$ is separable.
\end{proof}

\begin{theorem}\label{t:MS} For a pointed topological space $Y$ containing more than one point, and a topological space $X$ containing an isolated point, the following conditions are equivalent:
\begin{enumerate}
\item[\textup{(1)}] the function space $C_k'(X,Y)$ is metrizable and separable;
\item[\textup{(2)}]  $Y$ is a metrizable separable space, and $X$ is a hemi-$\dot\kappa_\w$-space with countable set $\dot X$ of isolated points.
\end{enumerate}
\end{theorem}

\begin{proof} To prove that $(1)\Ra(2)$, assume that the function space $C_k'(X,Y)$ is separable and metrizable.  By Theorem~\ref{t:M}, the space $Y$ is metrizable and $X$ is a hemi-$\dot\kappa_\w$-space. By Lemma~\ref{l:spread}, the set $\dot X$ is countable. Since the space $X$ has an isolated point $x$, the space $Y$ is separable, being the image of the separable space $C_k'(X,Y)$ under the continuous map $\delta_x:C_k'(X,Y)\to Y$, $\delta_x:f\mapsto f(x)$.
\smallskip

$(2)\Ra(1)$ Assume that $Y$ is a separable metrizable space, and $X$ is a hemi-$\dot\kappa_\w$-space with countable set $\dot X$ of isolated points. By Theorem~\ref{t:M}, the function space $C_k'(X,Y)$ is metrizable and by Lemma~\ref{l:separable}, it is separable.
\end{proof}

Theorem~\ref{t:MS} and Lemmas~\ref{l:M1a}, \ref{l:MS2} imply the following characterization.

\begin{corollary}\label{c:AMS} For a pointed metrizable space $Y\ne\{*_Y\}$ and a topological space $X$ containing an isolated point, the following conditions are equivalent:
\begin{enumerate}
\item[\textup{(1)}] the function space $C_k'(X,Y)$ is metrizable and separable;
\item[\textup{(2)}] $C_k'(X,Y)$ contains a dense separable first-countable subspace;
\item[\textup{(3)}]  $Y$ is separable and $X$ is a hemi-$\dot\kappa_\w$-space with countable set $\dot X$ of isolated points.
\end{enumerate}
\end{corollary}

\section{The (almost) complete-metrizability of the function spaces $C_k'(X,Y)$}\label{s:CM}

In this section we characterize function spaces $C_k'(X,Y)$ which are (almost) complete-metrizable or (almost) Polish.

\begin{theorem}\label{t:CM} For a pointed topological space $Y$ containing more than one point, and a topological space $X$ containing an isolated point, the following conditions are equivalent:
\begin{enumerate}
\item[\textup{(1)}] the function space $C_k'(X,Y)$ is complete-metrizable;
\item[\textup{(2)}]  the spaces $Y$ and $C_k'(X,2)$ are complete-metrizable;
\item[\textup{(3)}] the space $Y$ is complete-metrizable and $X$ is a $\dot\kappa_\w$-space.
\end{enumerate}
\end{theorem}

\begin{proof} $(1)\Ra(2)$. Assume that the function space $C_k(X,Y)$ is complete-metrizable. Take any isolated point $x\in\dot X$ and observe that the map $$H:C_k'(X,Y)\to Y\times C_k'(X\setminus\{x\},Y),\;\;H:f\mapsto(f(x),f{\restriction}X\setminus\{x\}),$$
is a homeomorphism. Then the space $Y$ is complete-metrizable, being homeomorphic to a closed subspace of the complete-metrizable space $C_k'(X,Y)$.

Next, choose any point $y\in Y\setminus \{*_Y\}$ and consider the bijective map $e:\{0,1\}\to\{*_Y,y\}$ such that $e(0)=*_Y$ and $e(1)=y$. The map $e$ induces a closed topological embedding $e^*:C_k'(X,2)\to C_k'(X,Y)$, $e^*:f\mapsto e\circ f$. Then the space $C_k'(X,2)$ is complete-metrizable, as a closed subspace of the complete-metrizable  space $C'_k(X,Y)$.
\smallskip

$(2)\Ra(3)$ Assume that the space $C_k'(X,2)$ is complete-metrizable. We can endow the doubleton $2=\{0,1\}$ with the group operation having $0$ as its neutral element and consider $C_k'(X,2)$ as an abelian metrizable topological group. Being complete-metrizable, this topological group is  complete in its uniformity (of uniform convergence on compacta).

Since the complete-metrizable space $C_k'(X,2)$ is first-countable, there exist an increasing sequence $(K_n)_{n\in\w}$ of compact subsets of $X$ such that the sets $$[K_n;\{0\}]:=\big\{f\in C_k'(X,Y):f(K_n)\subset \{0\}\big\},\quad n\in\w,$$
form a neighborhood base at the constant function $e:X\to\{0\}$, which is the neutral element of  the group $C_k'(X,2)$.

We claim that the sequence $(K_n)_{n\in\w}$ witnesses that $X$ is a $\dot\kappa_\w$-space. First observe that for any compact set $K\subset X$ we can find $n\in\w$ such that $[K_n,\{0\}]\subset [K,\{0\}]$. The latter inclusion implies that $K\cap \dot X\subset K_n$. So, $\dot X\subset\bigcup_{n\in\w}K_n$.

Next, we shall prove that a subset $D\subset\dot X$ is closed in $X$ if  $D\cap K_n$ is finite for every $n$.  For every $n\in\w$ consider the continuous function $f_n\in C'_k(X,2)$ defined by $f_n^{-1}(1)=D\cap K_n$. We claim that the sequence $(f_n)_{n\in\w}$ is Cauchy in the uniformity of uniform convergence on compacta. Given a compact set $K\subset X$, we should find $n\in\w$ such that $f_m{\restriction}K=f_n{\restriction}K$ for all $m\ge n$. Consider the open neighborhood $[K;\{0\}]\subset C_k'(X,2)$ of the constant function $e$ and find $n\in\w$ such that $[K_n;\{0\}]\subset [K;\{0\}]$. The latter inclusion implies $K\cap \dot X\subset  K_n$. Then for any $m\ge n$ we get $D\cap K_n\cap K=D\cap K=D\cap K_m\cap K$, which implies $f_n{\restriction}K=f_m{\restriction}K$. By the completeness of the group $C_k'(X,2)$, the Cauchy sequence $(f_n)_{n\in\w}$ converges to a continuous function $f_\infty:X\to 2$. The only choice for this limit is the characteristic function of the set $D$, which implies that the set $D=f_\infty^{-1}(1)$ is closed in $X$ by the continuity of $f_\infty$.
\smallskip

$(3)\Ra(1)$ Assume that the space $Y$ is complete-metrizable and $X$ is a $\dot\kappa_\w$-space. Let $(K_n)_{n\in\w}$ be an increasing sequence of compact sets, witnessing that $X$ is a $\dot\kappa_\w$-space. For every $n\in\w$, the compactness of $K_n$ and the complete-metrizability of $Y$ imply the complete-metrizability of the function space $C_k(K_n,Y)$. Then the product $\prod_{n\in\w}C_k(K_n,Y)$ is complete-metrizable as well. The proof of Lemma~\ref{l:M2} implies that the map $$\delta:C_k'(X,Y)\to\prod_{n\in\w}C_k(K_n,Y),\;\;\delta:f\mapsto (f{\restriction}K_n)_{n\in\w},$$
is a topological embedding.

We claim that the image $\delta(C_k'(X,Y))$ is a closed subset of $\prod_{k\in\w}C_k(X_n,Y)$. Take any element $(f_n)_{n\in\w}\in\overline{\delta(C'_k(X,Y))}\subset \prod_{n\in\w}C_k(K_n,Y)$. It follows from  $(f_n)_{n\in\w}\in\overline{\delta(C'_k(X,Y))}$ that for every $n\le m$ the restriction $f_m{\restriction}K_n$ coincides with the function $f_n$ and moreover $f_n(K_n\cap X')\subset\{*_Y\}$. So, we can define a function $f:X\to Y$ by $f(X')\subset\{*_Y\}$ and $f{\restriction}K_n=f_n$ for every $n\in\w$. We claim that the function $f$ is continuous. It suffices to prove the continuity of $f$ at each non-isolated point $x'\in X'$. Assuming that $f$ is discontinuous at $x'$, we can find an open neighborhood $U_*\subset Y$ of the point $*_Y=f(x')$ whose
preimage $f^{-1}(U_*)$ is not a neighborhood of $x'$, which means that the set $D=X\setminus f^{-1}(U_*)\subset\dot X$ contains the point $x'$ in its closure.

Since the sequence $(K_n)_{n\in\w}$ witnesses that $X$ is a $\dot\kappa_\w$-space, for some $n\in\w$ the intersection $D\cap K_n$ is infinite. On the other  hand, the set $$D\cap K_n=K_n\setminus f^{-1}(U_*)=K_n\setminus (f{\restriction}K_n)^{-1}(U_*)=K_n\setminus f_n^{-1}(U_*)\subset K_n\cap\dot X$$is closed in $K_n$ by the continuity of $f_n$. Being a closed discrete subset of the compact space $K_n$, the set $D\cap K_n$ is finite, which contradicts the choice of $n$. This contradiction shows that the function $f$ is continuous and hence $(f_n)_{n\in\w}=(f{\restriction}K_n)_{n\in\w}=\delta(f)\in\delta(C'_k(X,Y))$. So, the set $\overline{\delta(C_k'(X,Y))}$ is closed in $\prod_{n\in\w}C_k(K_n,Y)$ and the space $C_k'(X,Y)$ is complete-metrizable, being homeomorphic to the closed subspace $\delta(C_k'(X,Y))$ of the  complete-metrizable space $\prod_{n\in\w}C_k(K_n,Y)$.
\end{proof}

Theorems~\ref{t:CM} and \ref{t:MS} imply the following characterization of Polish function spaces $C_k'(X,Y)$.

\begin{theorem}\label{t:P} For a pointed topological space $Y$ that contains more than one point, and a topological space $X$ containing an isolated point, the following conditions are equivalent:
\begin{enumerate}
\item the function space $C_k'(X,Y)$ is Polish;
\item  the spaces $Y$ and $C_k'(X,2)$ are Polish;
\item the space $Y$ is Polish and $X$ is a $\dot\kappa_\w$-space with countable set $\dot X$ of isolated points.
\end{enumerate}
\end{theorem}

\begin{theorem} For  a topological space $X$ and a pointed Polish space  $Y$ the function space $C_k'(X,Y)$ is Polish if and only if $C_k'(X,Y)$ is a Choquet space with countable spread.
\end{theorem}

\begin{proof} The ``only if'' part is trivial. To prove the ``if'' part, assume that the function space $C_k'(X,Y)$ is Choquet and has countable spread. If $Y=\{*_Y\}$, then the space $C_k'(X,Y)$ is Polish, being a singleton. So, we assume that $Y\ne\{*_Y\}$. By Theorem~\ref{t:C}, the space $X$ has $\WDMOP$ and by Lemma~\ref{l:spread}, the set $\dot X$ is at most countable. By Theorem~\ref{t:W=k}, $X$ is a $\dot \kappa_\w$-space and by Theorem~\ref{t:P}, the function space $C_k'(X,Y)$ is Polish.
\end{proof}

Finally, we characterize pairs $X,Y$ for which the function space $C_k'(X,Y)$ is almost complete-metrizable or almost Polish.

\begin{theorem}\label{t:ACM} For a pointed metrizable space $Y\ne\{*_Y\}$ and a topological space $X$ with $\dot X\ne \emptyset$, the following conditions are equivalent:
\begin{enumerate}
\item[\textup{(1)}] the function space $C_k'(X,Y)$ is almost complete-metrizable;
\item[\textup{(2)}] the space $Y$ is almost complete-metrizable and $X$ is a $\dot\kappa_\w$-space.
\end{enumerate}
\end{theorem}

\begin{proof} $(1)\Ra(2)$ If $C'_k(X,Y)$ is almost complete-metrizable, then it is Choquet and by Lemma~\ref{l:C1}, $X$ has $\WDMOP$.
By our assumption, the space $X$ contains an isolated point $x$. It is easy to see that the map $\delta_x:C_k'(X,Y)\to Y$, $\delta_x:f\mapsto f(x)$, is an open surjection, so the space $Y$ is Choquet.
Being metrizable, the Choquet space $Y$ is almost complete-metrizable, see \cite[8.17]{Ke} or \cite[7.3]{BanP}. Being almost complete-metrizable, the space $C_k'(X,Y)$ contains a dense first-countable subspace. By Lemma~\ref{l:M1a}, $X$ is a hemi-$\dot\kappa_\w$-space. By Proposition~\ref{p:kappa}(5), $X$ is a $\dot\kappa_\w$-space.
\smallskip

$(2)\Ra(1)$ Assume that the space $Y$ is almost complete-metrizable and $X$ is a $\dot\kappa_\w$-space. Then $Y$ contains a dense complete-metrizable subspace $M\subset Y$. Let $\tilde Y$ be any complete-metrizable space containing $Y$ as a dense subspace. By \cite[3.11]{Ke}, the complete-metrizable space $M$ is a $G_\delta$-set in $\tilde Y$. Since the singleton $\{*_Y\}$ is a $G_\delta$-subset of $\tilde Y$, the union $M\cup\{*_Y\}$ is a $G_\delta$-set in $\tilde Y$. By \cite[3.11]{Ke}, the space $M\cup\{*_Y\}$ is complete-metrizable. Replacing $M$ by $M\cup\{*_Y\}$, we can assume that $*_Y\in M$. By Theorem~\ref{t:CM}, the function space $C_k'(X,M)$ is complete-metrizable. Since $C_k'(X,M)$ is a dense subspace in $C_k'(X,Y)$, the space $C_k'(X,Y)$ is almost complete-metrizable.
\end{proof}

\begin{theorem}\label{t:AP} For a pointed metrizable space $Y\ne\{*_Y\}$ and a topological space $X$ with $\dot X\ne \emptyset$, the following conditions are equivalent:
\begin{enumerate}
\item[\textup{(1)}] the function space $C_k'(X,Y)$ is almost Polish;
\item[\textup{(2)}] the space $Y$ is almost Polish and $X$ is a $\dot\kappa_\w$-space with countable set of isolated points.
\end{enumerate}
\end{theorem}

\begin{proof} $(1)\Ra(2)$ If $C_k'(X,Y)$ is almost Polish, then $C_k'(X,Y)$ is  almost complete-metrizable. Moreover, $C_k'(X,Y)$ is separable and hence has countable cellularity. By Theorem~\ref{t:ACM}, the space $Y$ is almost complete-metrizable and $X$ is a $\dot\kappa_\w$-space. By Lemma~\ref{l:MS2}, the $\dot\kappa_\w$-space has countable set $\dot X$ of isolated points.

It remains to prove that the space $Y$ is almost Polish. We already know that $Y$ is almost complete-metrizable and hence $Y$ contains a dense complete-metrizable subset $M\subset Y$. By \cite[3.11]{Ke}, $M$ is a $G_\delta$-set in $Y$.
By our assumption, the space $X$ contains an isolated point $x$. It is easy to see that the map $\delta_x:C_k'(X,Y)\to Y$, $\delta_x:f\mapsto f(x)$, is an open surjection. This implies that the preimage $\delta_x^{-1}(M)$ is a dense $G_\delta$-set in $C_k'(X,Y)$. By our assumption, the space $C_k'(X,Y)$ is almost Polish and hence contains a dense Polish subspace $P$. Since the complement $C_k'(X,Y)\setminus \delta_x^{-1}(M)$ is a meager $F_\sigma$-set in $C'_k(X,Y)$, the set $P\setminus \delta_x^{-1}(M)$ is a meager $F_\sigma$-set in the Polish space $P$ and its complement $P\cap\delta^{-1}_x(M)$ is a dense $G_\delta$-set in $P$, by the classical Baire Theorem. Then $\delta_x(P\cap\delta^{-1}_x(M))$ is a dense separable set in $M$, which implies that the complete-metrizable space $M$ is separable and hence Polish. Consequently, the space $Y$ is almost Polish.
\smallskip

$(2)\Ra(1)$ Assume that the space $Y$ is almost Polish and $X$ is a $\dot\kappa_\w$-space with countable set of isolated points. Let $P$ be a dense Polish subspace in $Y$. Replacing $P$ by $P\cup\{*_Y\}$, we can assume that $*_Y\in P$. By Theorem~\ref{t:P}, the function space $C_k'(X,P)$ is Polish. Since the subspace $C_k'(X,P)$ is dense in $C_k'(X,Y)$, the space $C_k'(X,Y)$ is almost Polish.
\end{proof}

\section{Countable networks in function spaces $C_k'(X,Y)$}\label{s:N}

A family $\mathcal N$ of subsets of a topological space $X$ is called
\begin{itemize}
\item a {\em network} if for any open set $U\subset X$ and point $x\in U$ there exists a set $N\in\mathcal N$ such that $x\in N\subset U$;
\item a {\em $cs^*$-network} if for any open set $U\subset X$ and sequence $\{x_n\}_{n\in\w}\subset X$ that converges to a point $x_\infty\in U$ there is a set $N\in\mathcal N$ such that $x_\infty\in N\subset U$ and $N$ contains infinitely many points $x_n$, $n\in\w$;
\item a {\em $k$-network} if for any open set $U\subset X$ and compact subset $K\subset U$  there exists a finite subfamily $\F\subset\mathcal N$ such that $K\subset\bigcup\mathcal F\subset U$;
\item a {\em $\dot\kappa$-network} if for any compact set $K\subset X$ and any closed subset $D\subset \dot X\setminus K$ of $X$ there exists a finite subfamily  $\mathcal F\subset\mathcal N$ such that $K\cap\dot X\subset \bigcup\F\subset X\setminus D$.
\end{itemize}

It is clear that for any family $\mathcal N$ we have the implications
$$\mbox{$k$-network $\Ra$ $cs^*$-network $\Ra$ network}.$$

In case of countable networks, we have the following equivalence, which was proved for Hausdorff spaces in \cite{Guthrie}.

\begin{lemma}\label{l:cs=>k} Each countable $cs^*$-network is a $k$-network.
\end{lemma}

\begin{proof} Let $\mathcal N$ be a countable $cs^*$-network for a topological space $X$. To prove that it is a $k$-network, fix an open set $W\subset X$ and a compact subset $K\subset W$. Consider the countable subfamily $\mathcal N_W:=\{N\in\mathcal N:N\subset W\}$ and write it as $\mathcal N_W=\{N_k\}_{k\in\w}$. We claim that $K\subset\bigcup_{i\le n}N_i$ for some $n\in\w$. In the opposite case we can find a sequence of points $x_n\in K\setminus\bigcup_{i\le n}N_i$. Observe that the family $\{N\cap K:N\in\mathcal N\}$ is a countable network for the compact space $K$, which implies that $K$ is hereditarily Lindel\"of and hence sequentially compact, by a result of Alas and Wilson \cite{AW}. This allows us to find an increasing number sequence $(k_n)_{n\in\w}$ such that the subsequence $(x_{k_n})_{n\in\w}$ of the sequence $(x_k)_{k\in\w}$ converges to some point $x_\infty \in K$. Since $\mathcal N$ is a $cs^*$-network, there exists a set $N\in\mathcal N$ such that $N\subset W$ and  the set $\Omega=\{n\in\w:x_{k_n}\in N\}$ is infinite. Since $N\in\mathcal N_W$, there exists a number $m\in\w$ such that $N=N_m$ and then $x_i\notin N_m=N$ for all $i\ge m$. In particular, $\Omega\subset\{n\in\w:k_n<m\}$ is finite, which contradicts the choice of $N$.   This contradiction shows that $K\subset\bigcup_{i\le n}N_i\subset W$ for some $n$, which means that the family $\mathcal N$ is a $k$-network.
\end{proof}

Now we prove some results on networks in the function spaces $C_k'(X,Y)$.

\begin{lemma}\label{l:N1} Let $Y$ be a pointed topological space whose distinguished point $*_Y$ has a neighborhood $U_*\ne Y$. A topological space $X$ has a countable $\dot \kappa$-network if the function space $C_k'(X,Y)$ has a countable network.
\end{lemma}

\begin{proof} Let $\mathcal N$ be a countable network of the function space $C_k'(X,Y)$. For every $N\in\mathcal N$ let $N^*=\{x\in X:N\subset[\{x\};U_*]\}$, where $[\{x\};U_*]=\{f\in C_k'(X,Y):f(x)\in U_*\}$. We claim that the countable family $\mathcal N^*=\{N^*:N\in\mathcal N\}$ is a $\dot\kappa$-network for the space $X$.

Given a compact set $K\subset X$ and a closed subset $D\subset \dot X\setminus K$ of $X$, it suffices to find a set $N\in\mathcal N$ such that $K\subset N^*\subset X\setminus D$. Fix any point $y\in Y\setminus U_*$. Taking into account that the closed set $D\subset X$ consists of isolated points of $X$, we conclude that $D$ is clopen in $X$. Then its characteristic function $\chi:X\to\{*_Y,y\}$ defined by $\chi^{-1}(y)=D$ is continuous and the set $[K,U_*]$ is a neighborhood of $\chi$ in $C_k'(X,Y)$. Since $\mathcal N$ is a network of the topology of $C_k'(X,Y)$, there exists $N\in\mathcal N$ such that $\chi\in N\subset [K,U_*]$.   Observe that for each point $x\in K$ and any function $f\in N\subset [K,U^*]$ we get $f(x)\in U_*$ and hence $f\in[\{x\},U_*]$, which implies that $N\subset[\{x\};U_*]$ and hence $K\subset N^*$. On the other hand, for every $x\in D$ the function $\chi\in N$ does not belong to $[\{x\};U_*]$, which implies $N\not\subset [\{x\};U_*]$ and hence $x\notin N^*$. This means that $N^*\subset X\setminus D$.
\end{proof}

\begin{lemma}\label{l:N2} For any topological space $X$ with a countable $\dot\kappa$-network and any pointed topological space $Y$ with a countable base, the function space $C_k'(X,Y)$ has a countable network $\mathcal N$ (which is a countable $k$-network for $C_k'(X,Y)$ if $X$ is a $\w$-$\dot\kappa$-space).
\end{lemma}

\begin{proof} Let $\mathcal N_X$ be a countable $\dot\kappa$-network for the space $X$ and $\mathcal B_Y$ be a countable base for the space $Y$. We lose no generality assuming that the family $\mathcal N_X$ is closed under finite  unions. By the definition of a $\dot\kappa$-network, for any point $x\in\dot X$ there exists a set $N_x\in\mathcal N_X$ such that $\{x\}\subset N_x\subset X\setminus(X\setminus\{x\})=\{x\}$, which implies that $N_x=\{x\}$ and hence the family $\mathcal N_X$ contains all singletons $\{x\}\subset \dot X$. Now the countability of $\mathcal N_X$ implies the countability of the set $\dot X$.

Consider the family $\mathcal Q'$ of quadruples $(N,U,F,u)$ where $N\in\mathcal N_X$, $U\in\mathcal B_Y$ is a neighborhood of $*_Y$, $F\subset \dot X\setminus N$ is a finite subset, and $u:\dot X\to\mathcal B_Y$ is a function such that $u(x)=Y$ for all $x\in\dot X\setminus F$. Since the sets $\mathcal N_X$, $\dot X$, $\mathcal B_Y$ are countable, so is the family $\mathcal Q'$. For any quadruple $(N,U,F,u)\in\mathcal Q'$ consider the set
$$[N;U|F;u]:=\{f\in C_k'(X,Y):f(N)\subset U\}\cap\bigcap_{x\in F}\{f\in C_k'(X,Y):f(x)\in u(x)\}.$$
We claim that the countable family $\mathcal N=\big\{[N;U|F;u]:(N,U,F,u)\in\mathcal Q'\big\}$ is a network for the space $C_k'(X,Y)$.

Take any function $f\in C_k'(X,Y)$ and a neighborhood $O_f\subset C_k'(X,Y)$ of $f$. By Lemma~\ref{l:base}, there exists a quadruple $(K,U,F,u)\in\mathcal Q$ such that $f\in[K;U|F;u]\subset O_f$.  It follows that $K\subset f^{-1}(U)$ and the set $D=X\setminus f^{-1}(U)\subset\dot X\setminus K$ is closed in $X$. By the definition of the $\dot\kappa$-network $\mathcal N_X$, there exists a finite subfamily $\mathcal F\subset\mathcal N_X$ such that $K\cap\dot X\subset\bigcup\mathcal F\subset X\setminus (D\cup F)=f^{-1}(U_*)\setminus F$.  Since the family $\mathcal N_X$ is closed under finite unions, the set $N=\bigcup\mathcal F$ belongs to $N$. Now it is easy to see that
$$f\in[N;U|F;u]\subset [K;U|F;u]\subset O_f.$$

Next, assuming that $X$ is an $\w$-$\dot\kappa$-space, we shall prove that the family $\mathcal N$ is a $cs^*$-network for $C_k'(X,Y)$.
Given any open set $W\subset C_k'(X,Y)$ and a sequence of functions $\{f_n\}_{n\in\w}\subset W$ that converge to a function $f_\infty\in W$, we should find a quadruple $(N,U,F,u)\in\mathcal Q'$ such that $f_\infty\in [N;U|F;u]\subset W$ and $[N;U|F;u]$ contains infinitely many functions $f_n$. By Lemma~\ref{l:base}, there exists a quadruple $(K,U,F,u)\in\mathcal Q$ such that $f_\infty\in [K;U|F;u]\subset W$. Since $[K;U|F;u]$ is an open neighborhood of $f_\infty$, we can find a number $m\in\w$ such that $\{f_n\}_{n\ge m}\subset [K;U|F;u]$. We claim that the set $D=\bigcup_{n\in\w}f_n^{-1}(Y\setminus U)\subset\dot X$ is closed in $X$. Taking into account $X$ is an $\w$-$\dot\kappa$-space with countable set $\dot X$ of isolated points, we conclude that $X$ is a $\dot\kappa$-space. Assuming that the set $D$ is not closed in $X$, we can find a compact set $C\subset X$ such that $C\cap D$ is infinite. Observe that the set $C\setminus f_\infty^{-1}(U)$ is finite (being discrete and closed in the compact space $C$). So, we can replace $C$ by the compact set $C\cap f_\infty^{-1}(U)$ and assume that $f_\infty(C)\subset U$. Since the sequence $(f_n)_{n\in\w}$ converges to $f_\infty\in[C;U]$, there exists a number $l\in\w$ such that $f_n\in[C;U]$ for all $n\ge l$. Then $C\cap D=\bigcup_{n<l}C\setminus f_n^{-1}(U)$ is finite, being a finite union of finite sets $C\setminus f_n^{-1}(U)$. This contradiction finishes the proof of the closedness of the set $D$. Since $\mathcal N_X$ is a $\dot\kappa$-network for $X$, there exists a set $N\in\mathcal N_X$ such that $K\cap\dot X\subset N\subset X\setminus(D\cup F)$. It is easy to see that
$$\{f_n\}_{n\ge m}\subset [N;U|F;u]\subset W.$$
Therefore, $\mathcal N$ is a countable $cs^*$-network. By Lemma~\ref{l:cs=>k}, $\mathcal N$ is a countable $k$-network for the function space $C_k'(X,Y)$.
 \end{proof}

 Lemmas~\ref{l:N1} and \ref{l:N2} imply the following characterization.

\begin{theorem}\label{t:N} For a pointed metrizable space $Y\ne\{*_Y\}$ and a topological space $X\ne X'$ the following conditions are equivalent:
\begin{enumerate}
\item[\textup{(1)}] the function space $C_k'(X,Y)$ has a countable network;
\item[\textup{(2)}] the space $Y$ is separable and the space $X$ has a countable $\dot\kappa$-network.
\end{enumerate}
If $X$ is a $\w$-$\dot\kappa$-space, then the conditions \textup{(1),(2)} are equivalent to
\begin{itemize}
\item[\textup{(3)}] the function space $C_k'(X,Y)$ has a countable $k$-network.
\end{itemize}
\end{theorem}

\begin{problem}\label{prob:N} Is the condition \textup{(3)} in Theorem~\ref{t:N} equivalent to the conditions \textup{(1)} and \textup{(2)} for any topological space $X$?
\end{problem}

\begin{theorem} For a topological space $X$ with $\DMOP$ the following conditions are equivalent:
\begin{enumerate}
\item[\textup{(1)}] $X$ has a countable $\dot\kappa$-network;
\item[\textup{(2)}] $X$ is a hemi-$\dot\kappa_\w$-space with countable set $\dot X$ of isolated points.
\end{enumerate}
\end{theorem}

\begin{proof} $(2)\Ra(1)$ First assume that $X$  is a hemi-$\dot\kappa_\w$-space with countable set $\dot X$ of isolated points, and find a sequence $(K_n)_{n\in\w}$ of compact sets witnessing that $X$ is hemi-$\dot\kappa_\w$-compact. Replacing each compact set $K_n$ by the union $\bigcup_{i\le n}K_i$, we can assume that $K_n\subset K_{n+1}$ for all $n\in\w$. We claim that the countable family $$\mathcal N=\{K_n\setminus F:n\in\w,\;F\in\mathcal F(\dot X)\}$$ is a $\dot\kappa$-network for $X$. Indeed, for any compact set $K\subset X$ and any closed subset $D\subset\dot X\setminus K$ of $X$, we can find $n\in\w$ with $K\cap \dot X\subset K_n$ and observe that the set $F=K_n\cap D$ is finite (being closed and discrete in the compact space $K_n$).  Then the set $N=K_n\setminus F\in\mathcal N$ has the required property: $K\cap\dot X\subset N\subset X\setminus D$.
\smallskip

$(1)\Ra(2)$ Assume that the space $X$ has a countable $\dot\kappa$-network. By Corollary~\ref{c:B} and Theorem~\ref{t:N}, the function space $C_k'(X,2)$ is Baire and has a countable network. By \cite{BH}, each Baire topological group with countable network is metrizable and separable. Applying this result to the space $C_k'(X,2)$ (carrying a structure of a topological group), we conclude that $C_k'(X,2)$ is metrizable and separable. By Theorem~\ref{t:MS}, $X$ is a hemi-$\dot\kappa_\w$-space with countable set of isolated points.
\end{proof}

\section{Recognizing $\infty$-meager function spaces $C_k'(X,Y)$}\label{s:mm}

Lemma~\ref{l:B1} implies that for any $*$-admissible pointed space $Y$ and any topological space $X$ that does not satisfy $\DMOP$ the function space $C_k'(X,Y)$ is meager. In this section we prove that the meagerness of $C_k'(X,Y)$ in this result can be improved to a stronger property, called the $\infty$-meagerness.

\begin{definition}
A subset $A$ of a topological space $X$ is called
\begin{itemize}
\item {\em $\infty$-dense in $X$} if for any compact Hausdorff space $K$ the subset $C_k(K,A)=\{f\in C_k(K,X):f(K)\subset A\}$ is dense in $C_k(K,X)$;
\item {\em $\infty$-codense in $X$ } if the complement $X\setminus A$ is $\infty$-dense in $X$;
\item {\em $\infty$-meager in $X$} if $A$ is contained in a countable union of  closed $\infty$-codense subsets of $X$.
\end{itemize}
A topological space $X$ is {\em $\infty$-meager} if it $X$ is an $\infty$-meager subset of $X$.
\end{definition}
It is easy to see that each closed $\infty$-codense set is nowhere dense, so each $\infty$-meager set is meager. For future applications, $\infty$-meager spaces are important as the $\infty$-meagerness implies the $\sigma Z$-space property, which is a key ingredient in many characterization results of Infinite-Dimensional Topology, see \cite{BCZ}, \cite{BRZ}, \cite{BP}, \cite{Chig}, \cite{vM1}, \cite{vM2}, \cite{Sak}.

\begin{theorem}\label{t:m=M}  If a topological space $X$ does not have $\DMOP$, then for any $*$-admissible pointed topological space $(Y,*_Y)$, the function space $C_k'(X,Y)$ is $\infty$-meager.
\end{theorem}

\begin{proof} Assuming that $X$ does not have $\DMOP$ and applying Lemma~\ref{l:B1}, we conclude that the function space $C_k'(X,2)$ is meager and hence can be written as the countable union $C_k'(X,2)=\bigcup_{n\in\w}M_n$ of closed nowhere dense subsets $M_n$ in $C_k'(X,2)$. 

Since the pointed space $Y$ is $*$-admissible, the point $*_Y$ has an open neighborhood $U_*\subset Y$ which is not dense in $Y$. Let $\chi:Y\to\{0,1\}$ be a  (unique) function such that $\chi^{-1}(0)=U_*$.  Observe that for every function $f\in C_k'(X,Y)$ the composition $\chi\circ f:X\to 2$ is a continuous function that belongs to the space $C_k'(X,2)$. So, for every $n\in\w$ we can consider the set $$Z_n:=\{f\in C_k'(X,Y):\chi\circ f\in M_n\}\subset C_k'(X,Y)$$ and its closure $\bar Z_n$ in $C_k'(X,Y)$. It is clear that $C_k'(X,Y)=\bigcup_{n\in\w}\bar Z_n$.

It remains to prove that each set $\bar Z_n$ is $\infty$-codense in $C_k'(X,Y)$. Given any compact Hausdorff space $K$, continuous map $\mu:K\to C_k'(X,Y)$, and neighborhood $O_\mu\subset C_k(K,C_k'(X,Y))$ of $\mu$, we need to find a continuous map $\mu'\in O_\mu$ such that $\mu'(K)\cap\bar Z_n=\emptyset$.

On the function space $C_k'(X,Y)$ consider the base $\mathcal B$ of the topology, consisting of the sets $[K_1;U_1]\cap\dots\cap [K_m;U_m]$ where $K_1,\dots,K_m$ are non-empty compact sets in $X$ and $U_1,\dots,U_m$ are non-empty open sets in $Y$.

We lose no generality assuming that the neighborhood $O_\mu$ is of basic form
$$O_\mu=\bigcap_{i=1}^m[K_i;B_i]$$where $K_1,\dots,K_m$ are compact sets in $K$ and $B_1,\dots,B_m\in\mathcal B$. For any $i\le m$ find a finite family $\K_i$ of non-empty compact sets in $X$ and a function $u_i:\K_i\to\tau_Y$ to the topology $\tau_Y$ of $Y$ such that $B_i=\bigcap_{\kappa\in\K_i}[\kappa;u_i(\kappa)]$. Consider the compact set $C:=\bigcup_{i=1}^m\bigcup\K_i$ and the open neighborhood
$$V_*=U_*\cap\bigcap_{i=1}^m\bigcap\{u_i(\kappa):\kappa\in\K_i,\;*_Y\in u_i(\kappa)\}$$
of $*_Y$.

For a point $z\in K$ it will be convenient to denote the function $\mu(z)\in C_k'(X,Y)$ by $\mu_z$. Since $\mu_z(C\cap X')\subset \mu_z(X')\subset\{*_Y\}\subset V_*$, there exists a finite set $F_z\subset C\cap \dot X$ such that $\mu_z(C\setminus F_z)\subset V_*$. Then $[C\setminus F_z;V_*]$ is an open neighborhood of the function $\mu_z$ in $C_k'(X,Y)$ and by the continuity of the map $\mu:K\to C_k'(X,Y)$, there exists an open neighborhood $O_z\subset K$ of $z$ such that $\mu(O_z)\subset [C\setminus F_z;V_*]$. By the compactness of $K$, the open cover $\{O_z:z\in K\}$ of $K$ has a finite subcover $\{O_z:z\in E\}$. Then for the finite set $F=\bigcup_{z\in E}F_z\subset C\cap\dot X$ we have the inclusion $\mu(K)\subset [C\setminus F;V_*]$.

Observe that the map
$$H:C_k'(X,2)\to 2^F\times C_k'(X\setminus F,2),\;\;H:f\mapsto (f{\restriction}F,f{\restriction}X\setminus F),$$is a homeomorphism. For every function $v:F\to 2$ consider the open embedding $e_v:C_k'(X\setminus F,2)\to C_k'(X,2)$ assigning to each function $f\in C_k'(X\setminus F,2)$ the function $e_v f:X\to 2$ such that $e_v f{\restriction}X\setminus F=f$ and $e_vf{\restriction}F=v$. The nowhere density of the closed set $M_n$ in $C_k'(X,2)$ implies that the finite union $$M=\bigcup_{v\in 2^F}e_v^{-1}(M_n)$$ is closed and nowhere dense in $C_k'(X\setminus F,2)$. So, we can choose a function $$\hbar\in [C\setminus F;\{0\}]\cap C_k'(X\setminus F,2)\setminus M$$ and using the definition of the compact-open topology on $C_k'(X\setminus F,2)$, find a compact set $A\subset X\setminus F$ such that  the neighborhood $$O_\hbar:=\{f\in C_k'(X\setminus F,2):f{\restriction}A=\hbar{\restriction}A\}$$ does not intersect the set $M$. Replacing $A$ by $A\cup(C\setminus F)$, we can assume that $C\setminus F\subset A$. It follows that the set $\bigcup_{v\in 2^F}e_v(O_\hbar)$ is disjoint with the set $M_n$. The continuity of the function $\hbar\in C_k'(X\setminus F,2)$ implies that the set $A_*:=A\cap \hbar^{-1}(0)$ is compact and has finite complement $A\setminus A_*$. Moreover, since $\hbar \in [C\setminus F;\{0\}]$, the set $A_*$ contains $C\setminus F$, which implies that $C\cap (A\setminus A_*)=\emptyset$.

Since the set $U_*$ is not dense in $Y$, there exists a point $y\in Y\setminus\bar U_*$. Now consider the map $\mu':K\to C'_k(X,Y)$ assigning to each $z\in K$ the function $\mu_z':X\to Y$ defined by the formula
$$\mu'_z(x)=\begin{cases}\mu_z(x)&\mbox{if $x\in F$}\\
y&\mbox{if $x\in A\setminus A_*$}\\
*_Y&\mbox{otherwise}.
\end{cases}
$$It is easy to see that the map $\mu':K\to C_k'(X,Y)$ is continuous and $\mu'\in O_\mu$. It remains to prove that $\mu'(K)\cap\bar Z_n=\emptyset$.
Observe that the set $\mu'(K)$ is contained in the open subset  $W:=[A_*;V_*]\cap [A\setminus A_*;Y\setminus\bar U_*]$ of $C'_k(X,Y)$.
So, it suffices to prove that $W\cap Z_n=\emptyset$. Take any function $w\in W$ and consider the function $v=\chi\circ w{\restriction}F$. The inclusion $w\in W$ yields the equality $\chi\circ w{\restriction}A=\hbar{\restriction}A$ implying the inclusion $\chi\circ w{\restriction}X\setminus F\subset O_\hbar\subset C_k'(X\setminus F,2)\setminus M$ and $\chi\circ w=e_v( \chi\circ w{\restriction}X\setminus F)\in e_v(O_\hbar)\subset C_k'(X,2)\setminus M_n$ and hence $w\notin Z_n$.
\end{proof}

\section{A dichotomy for analytic function spaces $C_k'(X,Y)$}\label{s:a}

A topological space $X$ is defined to be
 {\em analytic} if there exists a surjective continuous map $f:P\to X$, defined on some Polish space $P$. By an old result of Christensen \cite[Theorem 5.4]{Chris}, each Baire analytic commutative topological group is Polish. We use this fact to prove the following dichotomy, which is the main result of this section.

\begin{theorem}\label{t:dichotomy} Let $Y$ be a pointed  Polish space. If for some topological space $X$ the function space $C_k'(X,Y)$ is analytic, then $C_k'(X,Y)$ is either Polish or $\infty$-meager.
\end{theorem}

\begin{proof} If $Y=\{*_Y\}$, then $C_k'(X,Y)$ is a Polish space, being a singleton. So, assume that  $Y\ne\{*_Y\}$ and fix any point $y\in Y\setminus\{*_Y\}$. Let $e:2\to\{*_Y,y\}$ be the map defined by $e(0)=*_Y$ and $e(1)=y$. The map $e$ induces a closed embedding $e^*:C_k'(X,2)\to C_k'(X,Y)$, $e^*:f\mapsto e\circ f$. Now we see that the space $C_k'(X,2)$ is analytic, being homeomorphic to a closed subspace of the analytic space $C_k'(X,Y)$. Since $C_k'(X,2)$ carries a natural structure of a commutative topological group, we can apply Christensen's Theorem 5.4 \cite{Chris} and conclude that $C_k(X',2)$ is either Polish or meager.

If $C_k'(X,2)$ is Polish, then by Theorem~\ref{t:P}, $X$ is a $\dot \kappa_\w$-space with countable set $\dot X$ of isolated points. By Theorem~\ref{t:P}, the function space $C'_k(X,Y)$ is Polish.

If $C_k'(X,2)$ is meager, then by Corollary~\ref{c:B}, the space $X$ does not $\DMOP$ and by Theorem~\ref{t:m=M}, the space $C_k'(X,Y)$ is $\infty$-meager.
\end{proof}

Now we prove that the analyticity of a cosmic function space $C_k'(X,Y)$ is equivalent to the analyticity of the space $$C_p'(X,Y)=\big\{f\in C_p(X,Y):f(X')\subset\{*_Y\}\big\}\subset C_p(X,Y).$$ Here by $C_p(X,Y)$ we denote the space $C(X,Y)$, endowed with the topology of pointwise convergence.
This topology coincides with the Tychonoff product topology, inherited from $Y^X$.

A topological space $X$ is called {\em cosmic} if it is regular and has a countable network. A function $f:X\to Y$ between two topological spaces is called {\em Borel} if for any open set $U\subset Y$ the preimage $f^{-1}(U)$ is a Borel subset of $X$.

\begin{proposition} For a pointed regular topological space $Y$ and a topological space $X$, the following conditions are equivalent:
\begin{enumerate}
\item[\textup{(1)}] the space $C_k'(X,Y)$ is analytic;
\item[\textup{(2)}] the space $C_k'(X,Y)$ is cosmic and the space $C_p'(X,Y)$ is analytic.
\end{enumerate}
\end{proposition}

\begin{proof} If $Y=\{*_Y\}$, then both spaces $C_k'(X,Y)$ and $C_p'(X,Y)$ are singletons and the conditions $(1),(2)$ are satisfied and hence are equivalent.
So, we assume that $Y\ne\{*_Y\}$. The regularity of the space $Y$ implies the regularity of the spaces $C_p(X,Y)\subset Y^X$ and $C_k(X,Y)$, see \cite[3.4.13]{Eng}.
\smallskip

$(1)\Ra(2)$ If the space $C_k'(X,Y)$ is analytic, then it is cosmic, being a continuous image of some Polish space, see \cite[4.9]{Gru}. The space $C_p'(X,Y)$ is analytic, being a continuous image of the analytic space $C_k'(X,Y)$.
\smallskip

$(2)\Ra(1)$ Now assume that the space $C_k'(X,Y)$ is cosmic and the space $C_p'(X,Y)$ is analytic. Observe that the identity map $C_k'(X,Y)\to C_p'(X,Y)$ is continuous. We claim that the  identity map $C_p'(X,Y)\to C_k'(X,Y)$ is Borel. Since the space $C_k'(X,Y)$ is hereditarily Lindel\"of (being cosmic), it suffices to check that for any compact set $K\subset X$ and any open set the subbasic open set $[K;U]\subset C_k(X,Y)$ is Borel in $C_p'(X,Y)$. This is trivially true if $[K;U]$ is empty. So we assume that $[K;U]\ne\emptyset$.

The space $C_k'(X,Y)$, being cosmic, has countable spread. Applying Lemma~\ref{l:spread}, we conclude that the set $\dot X$ is at most countable.

If $K\cap X'=\emptyset$, then $K$ is finite and the set $[K;U]$ is open in $C_p'(X,Y)$. If $K\cap X'\ne\emptyset$, then $[K;U]\ne\emptyset$ implies $*_Y\in U$ and $$[K;U]=\bigcap_{x\in \dot X\cap K}[\{x\};U]$$ is a $G_\delta$-set in $C_p'(X,Y)$.

Therefore, the identity map $C_p'(X,Y)\to C'_k(X,Y)$ is Borel. By \cite{Banal}, a cosmic space is analytic if and only if it is a Borel image of a Polish space. This characterization implies that the cosmic space $C_k'(X,Y)$ is analytic, being a Borel image of the analytic space $C_p'(X,Y)$.
\end{proof}

\section{The interplay between the function spaces $C_k'(X,Y)$ and $C_k'(X/X',Y)$}\label{s:ccm}

Given a non-discrete topological space $X$, by $X/X'$ we denote the quotient topological space (with the set $X'$, collapsed to the point $\{X'\}\in X/X':=(X\setminus X')\cup\{X'\}$) and observe that $X/X'$ is a Hausdorff space with a unique non-isolated point $\{X'\}$. So, $(X/X')'$ is a singleton.

For every pointed topological space $(Y,*_Y)$ the quotient map $q:X\to X/X'$ induces a continuous bijective map
$$q^*:C_k'(X/X',Y)\to C_k'(X',Y),\;\;q^*:f\mapsto f\circ q.$$
This map is a homeomorphism if the quotient map $q$ is {\em compact-covering} in the sense that any compact set $K\subset X/X'$ is contained in the image $q(C)$ of some compact set $C\subset X$. In this case the investigation of the function space $C_k'(X,Y)$ can be reduced to studying the function space $C_k'(X/X',Y)$ over the space $X/X'$ with a unique non-isolated point.

However in many natural situations the quotient map $q:X\to X/X'$ is not compact-covering. In particular, it is not compact-covering for the space $X=\w_1$ of all countable ordinals endowed with the order topology. The reason is that the space $\w_1$ is pseudocompact but not compact. We recall that a Tychonoff space $X$ is {\em pseudocompact} if each continuous real-valued function on $X$ is bounded.

\begin{example} If $X$ is a non-compact pseudocompact space $X$ with dense set $\dot X$ of isolated points, then
\begin{enumerate}
\item the space $X$ does not has $\DMOP$;
\item the space $X/X'$ is compact;
\item the function space $C_k'(X,2)$ is meager but $C_k'(X/X',2)$ is discrete;
\item the quotient map $q:X\to X/X'$ is not compact-covering.
\end{enumerate}
\end{example}

\begin{proof} 1. To see that $X$ does not have $\DMOP$, observe that the family of singletons $\big\{\{x\}:x\in\dot X\big\}$ is moving off, but contains no infinite discrete subfamily (otherwise $X$ would admit an unbounded real-valued continuous function).
\smallskip

2. Since the pseudocompactness is preserved by continuous images, the quotient space $Q=X/X'$ is pseudocompact. Since $Q$ has only one non-isolated point, it is compact. Indeed, given any open cover $\U$ of $Q$, we can find an open set $U\in\U$ containing the unique non-isolated point $q$ of the space $X$. We claim that the complement $Q\setminus U$ is finite. Otherwise, we could find a sequence $(x_n)_{n\in\w}$ of pairwise distinct points in $Q\setminus U$ and define a continuous unbounded function $f:Z\to\IR$ by the formula
$$
f(x)=\begin{cases}n,&\mbox{if $x=x_n$ for some $n\in\IN$};\\
0,&\mbox{otherwise}.
\end{cases}
$$But the existence of such unbounded continuous function $f$ contradicts the pseudocompactness of the space $Q=q(X)$.
\smallskip

3. Since the space $X$ does not have $\DMOP$, the function space $C_k'(X,2)$ is meager by Theorem~\ref{t:B}. Since the space $X/X'$ is compact, the function space $C_k'(X/X',2)$ is discrete.
\smallskip

4. Assuming that the quotient map $q:X\to X/X'=Q$ is compact-covering, we could find a compact set $K\subset X$ such that $q(K)=Q$. Then $K$ contains all isolated points  of the space. Since the set $\dot X$ is dense in $X$, we conclude that $X\subset\overline{K}=K$ and the space $X$ is compact, which contradicts our assumptions.
\end{proof}

Now we shall characterize topological spaces $X$ for which the quotient map $q:X\to X/X'$ is compact-covering.

A topological space $X$ is called a {\em $\mu$-space} if each closed bounded subset of $X$ is compact. A subset $B$ of a topological space $X$ is {\em bounded} if for any  continuous  function $f:X\to\IR$
the image $f(B)$ is a bounded subset of the real line.  It is known \cite[6.9.7]{AT} that the class of $\mu$-spaces includes all Diedounn\'e-complete spaces, and consequently, all paracompact spaces \cite[8.5.13]{Eng} and all submetrizable spaces \cite[6.10.8]{AT}.

\begin{definition} A topological space $X$ is called a {\em $\dot\mu$-space} if for any subset $A\subset\dot X$ the following conditions are equivalent:
\begin{itemize}
\item $A$ is contained in a compact subset of $X$;
\item for any closed subset $D\subset\dot X$ of $X$ the intersection $A\cap D$ is finite.
\end{itemize}
\end{definition}

\begin{theorem}\label{t:q} Let $Y$ be a pointed topological space whose distinguished point $*_Y$ has an open neighborhood $U_*\ne Y$. For a  non-discrete topological space $X$ the following conditions are equivalent:
\begin{enumerate}
\item[\textup{(1)}] $X$ is a $\dot\mu$-space;
\item[\textup{(2)}]  the quotient map $q:X\to X/X'$ is compact-covering;
\item[\textup{(3)}] the map $q^*:C_k'(X/X',Y)\to C_k'(X,Y)$, $q^*:f\mapsto f\circ q$, is a homeomorphism.
\end{enumerate}
The conditions \textup{(1)--(3)} follow from the condition
\begin{itemize}
\item[\textup{(4)}] $X$ is a $\mu$-space.
\end{itemize}
\end{theorem}

\begin{proof} $(1)\Ra(2)$ Assume that $X$ is a $\dot\mu$-space.
Given any compact subset $K\subset X/X'$, observe that the set $A=q^{-1}(K)\setminus X'$ has finite intersection $A\cap D$ with any closed subset $D\subset \dot X$ of $X$ (because $K$ is compact and $q(X\setminus D)$ is an open neighborhood of the unique non-isolated point of $X/X'$). Since $X$ is a $\dot\mu$-space, the set $A$ is contained in some compact subset $\tilde K$ of $X$. Replacing $\tilde K$ by a larger compact set, we can assume that $\tilde K\cap X'\ne\emptyset$. Then $q(\tilde K)$ is a compact subset of $X/X'$, containing $K$, which means that $q$ is compact-covering.
\smallskip

$(2)\Ra(1)$ Assume that the quotient map $q$ is compact-covering. To prove that $X$ is a $\dot\mu$-space, take any set $A\subset\dot X$ that has finite intersection $A\cap D$ with any closed subset $D\subset\dot X$ of $X$. The definition of the quotient topology on $X/X'$ ensures that the subset $q(A\cup X')\subset X/X'$ is compact. Since $q$ is compact-covering, there exists a compact set $K\subset X$ such that $q(A\cup X')\subset q(K)$ and hence $A\subset K$.
\smallskip

$(2)\Ra(3)$ It is easy to see (from the definition of the compact-open topology) that  the compact-covering property of the quotient map $q:X\to X/X'$ implies that the continuous bijective map $q^*:C_k'(X/X',Y)\to C_k'(X,Y)$ is a homeomorphism.
\smallskip

$(3)\Ra(2)$ Assuming that $q^*:C_k'(X/X',Y)\to C_k(X,Y)$ is a homeomorphism, we shall prove that the map $q:X\to X/X'$ is compact-covering. By our assumption, the distinguished point $*_Y$ of $Y$ has an open neighborhood $U_*$ that does not contain some point $y\in Y$. Given a compact subset $K\subset X/X'$, consider the open neighborhood $[K;U_*]$ of the constant function $X/X'\to\{*_Y\}\subset Y$. Since the map $q^*$ is open, the image $q^*([K;U_*])$ is an open neighborhood of the constant function $X\to\{*_Y\}\subset Y$. Consequently, there exists a compact set $\tilde K\subset X$ and a neighborhood $U\subset Y$ of $*_Y$ such that $[\tilde K;U]\subset q^*([K;U^*])$.
Replacing $\tilde K$ by a larger compact set, we can assume that $\tilde K\cap X'\ne\emptyset$.

We claim that $K\subset q(\tilde K)$. Given any point  $x\in K$, we should prove that $x\in q(\tilde K)$. If $x$ is not isolated in $X/X'$, then $x\in q(X')=q(X'\cap\tilde K)$. So we assume that $x$ is isolated in $X/X'$ and hence $x=q(\tilde x)$ for a unique isolated point $\tilde x\in \dot X$. We claim that $\tilde x\in\tilde K$. In the opposite case, we can consider the function $\chi:X\to \{*_Y,y\}$ defined by $\chi^{-1}(y)=\{\tilde x\}$, and observe that $\chi\in [\tilde K;U]\subset q^{*}([K;U_*])$, which implies that $y\in U_*$. But this contradicts the choice of the point $y$. This contradiction shows that $\tilde x\in\tilde K$ and hence $x=q(\tilde x)\in q(\tilde K)$ and the map $q$ is compact-covering.
\smallskip

$(4)\Ra(1)$ Finally, assuming that $X$ is a $\mu$-space, we prove that it is a $\dot\mu$-space. Fix any subset $A\subset\dot X$ that has finite intersection $A\cap D$ with any closed subset $D\subset\dot X$ of $X$. Assuming that $\bar A$ is not compact in the $\mu$-space $X$, we conclude that the set $A$ is not bounded. So  there exists a continuous function $f:X\to\IR$ such that the set $f(A)$ is unbounded in the real line. Then we can choose a sequence $\{a_n\}_{n\in\w}\subset A$ such that $|f(a_n)|>n$ for every $n\in\w$. The continuity of the function $f$ ensures that the sequence $(a_n)_{n\in\w}$ has no accumulation points in $X$, which means that the set $D=\{a_n\}_{n\in\w}$ is closed in $X$ and hence has a finite intersection $D\cap A=D$ with $A$, which is not true. This contradiction shows that $X$ is a $\dot\mu$-space.
\end{proof}

Now we recall the definitions of some familiar properties of topological spaces that correspond to the properties introduced in Definition~\ref{d:kappa}.

\begin{definition}\label{d:k} A topological space $X$ is called
\begin{itemize}
\item a {\em $k$-space} if a set $F\subset X$ is closed if for any compact set $K\subset X$ the intersection $F\cap K$ is closed in $K$;
\item a {\em $k_\w$-space} if there exists a countable family $\K$ of compact sets in $X$ such that a set $F\subset X$ is closed if for any $K\in\K$ the intersection $F\cap K$ is closed in $K$;
\item {\em hemicompact} if there exists a countable family $\K$ of compact sets such that each compact subset of $X$ is contained in some set $K\in\K$.
\end{itemize}
\end{definition}

Comparing Definitions~\ref{d:kappa} and \ref{d:k}, we can see that a topological space with a unique non-isolated point $X$ is a $\dot\kappa$-space (resp. a $\dot\kappa_\w$-space, a hemi-$\dot\kappa_\w$-space) if and only if $X$ is a $k$-space (resp. a $k_\w$-space, a hemicompact space). Using these equivalences, we can prove the following characterization.

\begin{proposition}\label{p:eq} A $\dot\mu$-space $X$
\begin{enumerate}
\item[\textup{(1)}] has $\DMOP$ iff $X/X'$ has $\DMOP$
     iff $X/X'$ has $\MOP$;
\item[\textup{(2)}] has $\WDMOP$ iff $X/X'$ has $\WDMOP$;
\item[\textup{(3)}] is a $\dot\kappa$-space iff $X/X'$ is a $k$-space;
\item[\textup{(4)}] is a hemi-$\dot\kappa_\w$-space iff $X/X'$ is hemicompact;
\item[\textup{(5)}] is a $\dot\kappa_\w$-space iff $X/X'$ is a $k_\w$-space;
\item[\textup{(6)}] has a countable $\dot\kappa$-network iff $X$ has a countable $k$-network.
\end{enumerate}
\end{proposition}

\begin{proof}
Notice that a topological space with unique non-isolated point has $\DMOP$ if and only if it has $\MOP$, then (1) is follows from Theorems~\ref{t:B}. The implications (2),(4),(5) and (6) follow directly from Theorems~\ref{t:C}, \ref{t:M}, \ref{t:CM} and \ref{t:N}, respectively.

(3) Suppose $X$ is a $\dot\kappa$-space and $D\subset\dot X$ such that for any compact subset $K\subset X/X'$, the intersection $D\cap K$ is finite. Then for any compact subset $\tilde K\subset X$,  the intersection $D\cap q(\tilde K)$ is finite, which implies that $D\cap \tilde K$ is finite, so $D$ is closed in $X$ and hence closed in $X/X'$. Consequently, $X/X'$ is $\dot\kappa$-space.
On the other hand, suppose $D\subset \dot X$ such that for any compact subset $\tilde K\subset X$, the intersection $D\cap \tilde K$ is finite. Since $q$ is compact-covering, for every compact subset $K\subset X/X'$ there exists a compact subset
$\tilde K\subset X$ such that $K\subset q(\tilde K)$. So $D\cap K\subset D\cap q(\tilde K)$ is finite, and $D$
is closed in $X/X'$. It follows that $D$ is closed in $X$ by the definition of quotient map, and hence $X$ is $\dot\kappa$-space.
\end{proof}

Proposition~\ref{p:eq}(6) allows us to give a partial answer to Problem~\ref{prob:N}.

\begin{corollary} For any $\dot\mu$-space $X$ with a countable $\dot\kappa$-network $X$ and any regular space $Y$ with a countable $k$-network, the function space $C_k'(X,Y)$ has a countable $k$-network.
\end{corollary}

\begin{proof} By Theorem~\ref{t:q}, the quotient map $q:X\to X/X'$ is compact-covering and function spaces $C_k'(X,Y)$ and $C_k'(X/X',Y)$ are homeomorphic.
By Proposition~\ref{p:eq}(6), the space $X/X'$ has a countable $k$-network. Being a Hausdorff space with a unique non-isolated point, the space $X/X'$ is regular. By a result of Michael \cite{Michael} (see also \cite[11.5]{Gru}), the function space $C_k(X/X',Y)$ has a countable $k$-network. Then its subspace $C_k'(X,Y)$ has a countable $k$-network, too.
\end{proof}

\section{Characterizing stratifiable scattered spaces of finite scattered height}\label{s:ss}

In this section we characterize stratifiable spaces among scattered space of finite scattered height.

A subset $A$ of a topological space $X$ is called
\begin{itemize}
\item a {\em retract} of $X$ if there exists a continuous map $r:X\to A$ such that $r(a)=a$ for all $x\in A$ (this map $r$ is called a {\em retraction} of $X$ onto $A$);
\item a {\em $G_\delta$-retract} of $X$, if $A$ is a retract in $X$ and $A$ is a $G_\delta$-subset of $X$.
\end{itemize}

\begin{theorem}\label{t:s'} A non-discrete topological space $X$ is stratifiable if and only if the set $X'$ of non-isolated points of $X$ is a stratifiable space and $X'$ is a $G_\delta$-retract of $X$.
\end{theorem}

\begin{proof} First, assume that $X'$ is a stratifiable $G_\delta$-retract in $X$. Fix a retraction $r:X\to X'$ and write $X'$ as the intersection $X'=\bigcap_{n\in\w}W_n$ of a decreasing sequence $(W_n)_{n\in\w}$ of open sets in $X$.

By definition, each point $x\in X'$ of the stratifiable space $X'$ has a countable family $(U_n(x))_{n\in\w}$ of open neighborhoods such that each closed set $F\subset X'$ is equal to the intersection $\bigcap_{n\in\w}\overline{U_n[F]}$.

For every $x\in X\setminus X'$ put $W_n(x)=\{x\}$ for all $n\in\w$. Also for every $x\in X'$ and $n\in\w$ put $W_n(x)=W_n\cap r^{-1}(U_n(x))$. We claim that the system of neighborhoods $\big((W_n(x))_{n\in\w}\big)_{x\in X}$ witnesses that the space $X$ is stratifiable. Given any closed subset $F\subset X$, we should prove that  $F=\bigcap_{n\in\w}\overline{W_n[F]}$ where $W_n[F]=\bigcup_{x\in F}W_n(x)$.
Given any point $x\notin F$, we should find $n\in\w$ such that $x\notin\overline{W_n[F]}$.

If $x\notin X'$, then there exists $n\in\w$ such that $x\notin W_n$. Then the neighborhood $\{x\}$ of $x$ is disjoint with the set $W_n[F]\subset F\cup W_n$ and we are done.

If $x\in X'$, then $x\notin F\cap X'=\bigcap_{n\in\w}\overline{U_n[F\cap X']}$. So, there exists a number $n\in\w$ with $x\notin\overline{U_n[F\cap X']}$ and then $x\notin F\cup r^{-1}(\overline{U_n[F\cap X']})$. It remains to observe that $$W_n[F]=F\cup W_n[F\cap X']\subset F\cup r^{-1}(U_n[F\cap X'])\subset F\cup r^{-1}(\overline{U_n[F\cap X']})$$ and
$\overline{W_n[F]}\subset F\cup r^{-1}(\overline{U_n[F\cap X']})\subset X\setminus\{x\}$.
This completes the proof of the ``if'' part.
\smallskip

To prove the ``only if'' part, assume that the space  $X$ is stratifiable. Since the stratifiability is inherited by subspaces, $X'$ is stratifiable. By definition, each closed subset of a stratifiable space is a $G_\delta$-set in $X$, which implies that the closed subset $X'$ of $X$ is a $G_\delta$-set in $X$. It remains to prove that $X'$ is a retract of $X$. By definition of stratifiability, each point $x$ has a countable family neighborhoods $\{U_n(x)\}_{n\in\w}$ such that $F=\bigcap_{n\in\w}\overline{U_n[F]}$ for every closed subset $F\subset X$. We lose no generality assuming that $U_0(x)=X$ for any $x\in X$. Replacing each neighborhood $U_n(x)$ by the intersection $\bigcap_{i\le n}U_i(x)$, we can assume that $U_{n+1}(x)\subset U_n(x)$ for all $n\in\w$ and $x\in X$.

Since $X'=\bigcap_{n\in\w}\overline{U_n[X']}=\bigcap_{n\in\w}U_n[X']$, for every $x\in \dot X$ the set $N_x=\{n\in\w:x\in{U_n[X']}\}$ is finite and hence has the largest element $n_x$. Choose any point $r[x]\in X'$ with $x\in U_{n_x}(r[x])$. We claim that the function $r:X\to X'$ defined by $$r(x)=\begin{cases}r[x]&\mbox{if $x\in\dot X$};\\
x&\mbox{otherwise};
\end{cases}
$$
is a continuous retraction of $X$ onto $X'$.

It suffices to prove the continuity of $r$ at any point $x'\in X'$. Take any open neighborhood $V\subset X$ of $x'$. Then $X\setminus V=\bigcap_{n\in\w}\overline{U_n[X\setminus V]}$ and there exists a number $n\in\w$ such that $x'\notin \overline{U_n[X\setminus V]}$. We claim that $O_{x'}:=U_n(x')\setminus\overline{U_n[X\setminus V]}$ is an open neighborhood of $x'$ such that for any $x\in O_x'$ we get $r(x)\in V$.

If $x\in X'\cap O_{x'}$, then $r(x)=x\in O_{x'}\subset X\setminus U_n[X\setminus V]\subset X\setminus(X\setminus V)=V$.

If $x\in \dot X$, then the inclusion $x\in O_{x'}\subset U_n(x')\subset U_n[X']$ implies that $n_x\ge n$ and $x\in U_{n_x}(r(x))$. Assuming that $r(x)\in X\setminus V$, we conclude that $x\in U_{n_x}(r(x))\subset U_{n_x}[X\setminus V]\subset U_n[X\setminus V]$, which contradicts the choice of $x\in O_{x'}\subset X\setminus U_n[X\setminus V]$.
\end{proof}

\begin{theorem}\label{t:ss} A scattered space $X$ of finite scattered height is stratifiable if and only if for every $n<\hbar[X]$ the set $X^{[n]}$ is a $G_\delta$-retract in $X$.
\end{theorem}

\begin{proof} This characterization will be proved by induction on the scattered height $\hbar[X]$. If $\hbar[X]=0$, then the space $X$ is discrete and hence stratifiable.

Assume that for some $n\in\IN$ we have prove that a scattered space $X$ of scattered height $\hbar[X]<n$ is stratifiable if and only if the sets $X^{[k]}$, $k<\hbar[X]$, are $G_\delta$-retracts of $X$.

Let $X$ be a scattered space of scattered height $\hbar[X]=n$. If $X$ is stratifiable, then  $X'$ is a stratifiable $G_\delta$-retract of $X$, by Theorem~\ref{t:s'}. Since  the space $X'$ has scattered height $\hbar[X']<n$ we can apply the inductive assumption and conclude that for every $k<\hbar[X']$ the set $X^{[k]}$ is a $G_\delta$-retract in $X'$. Since $X'$ is a $G_\delta$-retract in $X$, the set $X^{[k]}$ is a $G_\delta$-retract in $X$.

Now assume that for every $k<\hbar[X]$ the set $X^{[k]}$ is a $G_\delta$-retract in $X$. So, there exists a retraction $r_k:X\to X^{[k]}$ and $X^{[k]}$ is a $G_\delta$-set in $X$. Then $X^{[k]}$ is a $G_\delta$-set in $X'$ and the restriction $r_k{\restriction}X':X'\to X^{[k]}$ is a retraction of $X'$ onto its closed subset $X^{[k]}$. So $X^{[k]}$ is a $G_\delta$-retract of $X'$. By the inductive assumption, the space $X'$ is stratifiable. Since $X'$ is a $G_\delta$-retract of $X$, we can apply Theorem~\ref{t:s'} and conclude that the space $X$ is stratifiable.
\end{proof}

\section{On function spaces with values in rectifiable pointed spaces}\label{s11}

A pointed topological space $Y$ is called a {\em rectifiable space} if there exists a homeomorphism $H:Y\times Y\to Y\times Y$ such that $H(\{y\}\times Y)=\{y\}\times Y$ and $H(y,y)=(y,*_Y)$ for every $y\in Y$. Each topological group $Y$ is rectifiable because of the homeomorphism $H:Y\times Y\to Y\times Y$, $H:(y,z)\mapsto (y,y^{-1}z)$. The unit sphere $S^7$ in the space of Cayley octonions is an example of a rectifiable pointed space, which is not homeomorphic to a topological group.
By \cite{Gul}, a rectifiable space is regular (and metrizable) if and only if it satisfies the separation axiom $T_0$ (and is first-countable). 
More information on rectifiable spaces can be found in \cite{Ar02}, \cite{BR}, \cite{Gul}.

\begin{proposition}\label{p:rectifiable} Let $Y$ be a rectifiable space and $X$ be a topological space. If the set $X'$  is a retract of $X$, then the function space $C_k(X,Y)$ is homeomorphic to $C_k(X',Y)\times C_k'(X,Y)$.
\end{proposition}

\begin{proof}  By the rectifiability of the pointed  space $Y$, there exists a homeomorphism $h:Y\times Y\to Y\times Y$ such that $h(\{y\}\times Y)=\{y\}\times Y$ and $h(y,y)=(y,*_Y)$ for every $y\in Y$. Let $\pr_2:Y\times Y\to Y$, $\pr_2:(x,y)\mapsto y$, be the projection onto the second coordinate.
Assume that $r:X\to X'$ is a retraction of $X$ onto $X'$.

We claim that the map $$H:C_k(X,Y)\to C_k(X')\times C_k'(X,Y),\;\;H:f\mapsto \big(f{\restriction}X',\pr_2\circ h (f\circ r,f)\big),$$
is a homeomorphism. Here $\pr_2\circ h(f\circ r,f):X\to Y$ is the map assigning to each $x\in X$ the element $\pr_2\circ h(f{\circ}r(x),f(x))\in Y$.
Observe that for any $x\in X'$ we get $$\pr_2\circ h(f\circ r(x),f(x))=\pr_2\circ h(f(x),f(x))=\pr_2(f(x),*_Y)=*_Y,$$ which means that $\pr_2\circ h (f\circ r,f)\in C_k'(X,Y)$.

The map $H$ has a continuous inverse $$H^{-1}:C_k(X',Y)\times C_k'(X,Y)\to C_k(X,Y),\;\;H^{-1}:(g,f)\mapsto \pr_2\circ h^{-1}(g\circ r,f),$$
where the map $\pr_2\circ h^{-1} (g\circ r,f):X\to Y$ assigns to each $x\in X$ the element $\pr_2\circ h^{-1}(g\circ r(x),f(x))\in Y$.

To show that $H^{-1}\circ H$ is the identity map of $C_k(X,Y)$, take any function $f\in C_k(X,Y)$ and observe that
$$
\begin{aligned}
H^{-1}\circ H(f)&=H^{-1}\big(f{\restriction}X',\pr_2\circ h (f\circ r,f)\big)=\pr_2\circ h^{-1} ((f{\restriction}X')\circ r,\pr_2\circ h(f\circ r,f))=\\
&=\pr_2\circ h^{-1}(f\circ r,\pr_2\circ h(f\circ r,f))=\pr_2\circ h^{-1}\circ h(f\circ r,f)=\pr_2(f\circ r,f)=f.
\end{aligned}
$$
Next, we prove that $H\circ H^{-1}$ is the identity map of $C_k(X',Y)\times C_k'(X,Y)$.  Given any pair $(\varphi,\psi)\in C_k(X',Y)\times C_k'(X,Y)$, consider the function $f=H^{-1}(\varphi,\psi)=\pr_2\circ h^{-1} (\varphi\circ r,\psi)$. We claim that $H(f)=(\varphi,\psi)$. Recall that $H(f)=\big(f{\restriction}X',\pr_2\circ h(f\circ r,f)\big)$. Then $$f{\restriction}X'=\pr_2\circ h^{-1}(\varphi\circ r,\psi){\restriction}X'=\pr_2\circ h^{-1}(\varphi,*_Y)=\pr_2(\varphi,\varphi)=\varphi.$$
On the other hand,
$$
\begin{aligned}
&(f\circ r,\pr_2\circ h(f\circ r,f))=h(f\circ r,f)=h\big(\pr_2\circ h^{-1}(\varphi\circ r,\psi)\circ r,\pr_2\circ h^{-1}(\varphi\circ r,\psi)\big)=\\
&=h\big(\pr_2\circ h^{-1}(\varphi\circ r,*_Y),\pr_2\circ h^{-1}(\varphi\circ r,\psi)\big)=h\big(\pr_2(\varphi\circ r,\varphi\circ r),\pr_2\circ h^{-1}(\varphi\circ r,\psi)\big)=\\
&=h\big(\varphi\circ r,\pr_2\circ h^{-1}(\varphi\circ r,\psi)\big)=h\circ h^{-1}(\varphi\circ r,\psi)=(\varphi\circ r,\psi)
\end{aligned}
$$and hence $\psi=\pr_2\circ h(f\circ r,f)$.
\end{proof}

\section{Function spaces $C_k(X,Y)$ for scattered $X$ and rectifiable $Y$}\label{s5}

In this section we shall prove (general versions of) the results announced in the introduction.

\begin{theorem}\label{t:sB} Let $X$ be a topological space such that $X'$ is a $G_\delta$-retract in $X$ and $Y\ne\{*_Y\}$ be a second-countable Choquet rectifiable $T_0$-space. Then the following conditions are equivalent:
\begin{enumerate}
\item[\textup{(1)}] the function space $C_k(X,Y)$ is Baire.
\item[\textup{(2)}] the space $X$ has $\DMOP$ and the function space $C_k(X',Y)$ is Baire.
\end{enumerate}
\end{theorem}

\begin{proof} By Proposition~\ref{p:rectifiable}, the function space $C_k(X,Y)$ is homeomorphic to $C_k(X',Y)\times C_k'(X,Y)$.  By \cite{Gul}, the second-countable rectifiable $T_0$-space $Y$ is metrizable and separable. Consequently, $Y$ is $*$-first-countable. Since $Y\ne\{*_Y\}$, the metrizable space $Y$ is $*$-admissible. 
\smallskip

 $(1)\Ra(2)$ Assume that the function space $C_k(X,Y)$ is Baire.  Since the Baire space $C_k(X,Y)$ is homeomorphic to $C_k'(X,Y)\times C_k(X',Y)$, the spaces $C_k'(X,Y)$ and  $C_k(X',Y)$ are Baire. By Corollary~\ref{c:B}, the space $X$ has $\DMOP$.
 \smallskip

$(2)\Ra(1)$ Assume that the space $X$ has $\DMOP$ and the function space $C_k(X',Y)$ is Baire. By Corollary~\ref{c:B}, the space $C_k'(X,Y)$ is Baire. By our assumption, $X'$ is a $G_\delta$-set in $X$. By Proposition~\ref{p:cellular}, the space $C_k'(X,Y)$ has countable cellularity. By Theorem~3.1 of \cite{LZ}, the product of a Baire space and a countably cellular Baire space  is Baire. Consequently, the product $C_k(X',Y)\times C_k'(X,Y)$ is Baire and so is its topological copy $C_k(X,Y)$.
\end{proof}

Applying Theorem~\ref{t:sB} inductively, we can prove the following characterization.

\begin{corollary}\label{c:nB} Let $n\in\IN$ and $X$ be a topological space such that for every $k\le n$ the set $X^{[k]}$ is a $G_\delta$-retract in $X$. Then for any second-countable Choquet rectifiable $T_0$-space $Y\ne\{*_Y\}$, the following conditions are equivalent:
\begin{enumerate}
\item[\textup{(1)}] the function space $C_k(X,Y)$ is Baire;
\item[\textup{(2)}] for every $i\le n$ the space $X^{[i]}$ has $\DMOP$ and the space $C_k(X^{[n]},Y)$ is Baire.
\end{enumerate}
\end{corollary}


Combining Corollary~\ref{c:nB} with Theorem~\ref{t:ss}, we obtain the following characterization.

\begin{corollary}\label{c:sB} For any stratifiable scattered space $X$ of finite scattered height and any second-countable Choquet rectifiable space $Y\ne\{*_Y\}$, the following conditions are equivalent:
\begin{enumerate}
\item[\textup{(1)}] the function space $C_k(X,Y)$ is Baire;
\item[\textup{(2)}] for every $k<\hbar[X]$ the space $X^{[k]}$ has $\DMOP$.
\end{enumerate}
\end{corollary}

\begin{corollary}\label{c:sBR} For a stratifiable scattered space $X$ of finite scattered height the following conditions are equivalent:
\begin{enumerate}
\item[\textup{(1)}] The function space $C_k(X)$ is Baire.
\item[\textup{(2)}] The space $X$ has $\MOP$.
\item[\textup{(3)}] For every $n<\hbar[X]$ the space $X^{[n]}$ has $\DMOP$.
\end{enumerate}
\end{corollary}

\begin{proof} $(1)\Ra(2)$ If the function space $C_k(X)$ is Baire, then the space $X$ has $\MOP$ by Theorem~2.1 in \cite{GMa}.
\smallskip

$(2)\Ra(3)$ If $X$ has $\MOP$, then each closed subspace of $X$ has $\MOP$. In particular, for every $k<\hbar[X]$ the closed subspace $X^{[k]}$ has $\MOP$ and hence has $\DMOP$ (because $\MOP$ implies $\DMOP$).
\smallskip

The implication $(3)\Ra(1)$ follows from Corollary~\ref{c:sB} since the real line is a rectifiable Choquet space.
\end{proof}

Analogous result can be proved for the Choquet property.

\begin{theorem}\label{t:sC} Let $X\ne X'$ be a topological space such that $X'$ is a retract of $X$ and $Y\ne\{*_Y\}$ be a second-countable rectifiable $T_0$-space. Then the following conditions are equivalent:
\begin{enumerate}
\item[\textup{(1)}] the function space $C_k(X,Y)$ is Choquet;
\item[\textup{(2)}] the spaces $Y$ and $C_k(X',Y)$ are Choquet and
the space $X$ has $\WDMOP$.
\end{enumerate}
\end{theorem}

\begin{proof} By Proposition~\ref{p:rectifiable}, the function space $C_k(X,Y)$ is homeomorphic to $C_k(X',Y)\times C_k'(X,Y)$.
\smallskip

 $(1)\Ra(2)$ Assume that the function space $C_k(X,Y)$ is Choquet.  Since the Choquet space $C_k(X,Y)$ is homeomorphic to $C_k'(X,Y)\times C_k(X',Y)$, the spaces $C_k'(X,Y)$ and  $C_k(X',Y)$ are Choquet, being open continuous images of $C_k(X,Y)$. By Theorem~\ref{t:C}, the space $X$ has $\WDMOP$.
 \smallskip

$(2)\Ra(1)$ Assume that $Y$ is Choquet, $X$ has $\WDMOP$ and  function space $C_k(X',Y)$ is Choquet. By Theorem~\ref{t:C}, the space $C_k'(X,Y)$ is Choquet. Then the space $C_k(X,Y)$ is Choquet, being homeomorphic to the product $C_k'(X,Y)\times C_k(X',Y)$ of two Choquet spaces.
\end{proof}

This theorem has the following corollaries.

\begin{corollary}\label{c:nC} Let $n\in\IN$ and $X\ne X'$ be a topological space such that for every $k\le n$ the set $X^{[k]}$ is a retract of $X$. Then for any second-countable rectifiable $T_0$-space $Y\ne\{*_Y\}$, the following conditions are equivalent:
\begin{enumerate}
\item[\textup{(1)}] the function space $C_k(X,Y)$ is Choquet;
\item[\textup{(2)}] the spaces $Y$ and $C_k(X^{[n]},Y)$ are Choquet, and for every $k\le n$ the space $X^{[k]}$ has $\WDMOP$.
\end{enumerate}
\end{corollary}

\begin{corollary}\label{c:sC} For any non-empty stratifiable scattered space of finite scattered height and any second-countable rectifiable $T_0$-space $Y\ne\{*_Y\}$, the following conditions are equivalent:
\begin{enumerate}
\item[\textup{(1)}] the function space $C_k(X,Y)$ is Choquet;
\item[\textup{(2)}] $Y$ is Choquet and for every $k<\hbar[X]$ the space $X^{[k]}$ has $\WDMOP$.
\end{enumerate}
\end{corollary}

\begin{corollary}\label{c:sCR} For any stratifiable scattered space $X$ of finite scattered height the following conditions are equivalent:
\begin{enumerate}
\item[\textup{(1)}] The function space $C_k(X)$ is Choquet.
\item[\textup{(2)}] For every $n<\hbar[X]$ the space $X^{[n]}$ has $\WDMOP$.
\end{enumerate}
\end{corollary}


\end{document}